\documentclass[a4paper,12pt]{amsart}
\usepackage{amsfonts,amsthm,amssymb,amsmath,amscd}
\usepackage[colorlinks=true,linkcolor=blue,citecolor=blue]{hyperref}
\usepackage{amsmath}
\usepackage{tikz-cd}

\usepackage{tabularx}
\usepackage{tcolorbox}
\usepackage[english]{babel}
\usepackage{amssymb,amsmath}
\usepackage{xcolor}

\newtheorem{Theo}{Theorem}[section]
\newtheorem{Prop}[Theo]{Proposition}
\newtheorem{Coro}[Theo]{Corollary}
\newtheorem{Lemm}[Theo]{Lemma}

\newtheorem{Rema}[Theo]{Remark}

\newcommand{\Hcal}{\mathcal{H}}
\newcommand{\T}{\mathbb{T}}
\newcommand{\Bcal}{\mathcal{B}}

\newcommand{\C}{\mathbb{C}}
\newcommand{\Z}{\mathbb{Z}}

\def\N{\mathbb{ N}}
\def\R{\mathbb{ R}}

\begin{document}

\title{Variants of a theorem of Helson on general Dirichlet series}

\author[Defant]{Andreas Defant}
\address[]{Andreas Defant\newline  Institut f\"{u}r Mathematik,\newline Carl von Ossietzky Universit\"at,\newline
26111 Oldenburg, Germany.
}
\email{defant@mathematik.uni-oldenburg.de}

\author[Schoolmann]{Ingo Schoolmann}
\address[]{Ingo Schoolmann\newline  Institut f\"{u}r Mathematik,\newline Carl von Ossietzky Universit\"at,\newline
26111 Oldenburg, Germany.
}
\email{ingo.schoolmann@uni-oldenburg.de}

\maketitle

\begin{abstract}
\noindent A result of Helson on general Dirichlet series $\sum a_{n} e^{-\lambda_{n}s}$ states that, whenever $(a_{n})$ is $2$-summable and  $\lambda=(\lambda_{n})$ satisfies a certain condition introduced by Bohr, then for almost all homomorphism $\omega \colon (\R,+) \to \T$ the Dirichlet series $\sum a_{n} \omega(\lambda_{n})e^{-\lambda_{n}s}$
converges on the open right half plane $[Re>0]$.
For ordinary Dirichlet series $\sum a_n n^{-s}$ Hedenmalm and Saksman related  this result with the famous Carleson-Hunt theorem on pointwise convergence of Fourier series, and Bayart extended
it within his theory of Hardy spaces $\mathcal{H}_p$ of such series.
The aim here is to prove variants of Helson's theorem  within our recent theory of Hardy spaces $\mathcal{H}_{p}(\lambda),\,1\le p \le \infty,$
of general Dirichlet series.
To be more precise, in the reflexive case $1 < p < \infty$ we extend Helson's result to Dirichlet series
in $\mathcal{H}_{p}(\lambda)$ without any further condition on the frequency $\lambda$, and in the non-reflexive case  $p=1$ to
  the wider class of frequencies satisfying the so-called Landau condition (more general than Bohr's condition). In both cases we add relevant maximal inequalities.
    Finally, we give several applications to the structure theory of
  Hardy spaces of general Dirichlet series.

\end{abstract}


\noindent
\renewcommand{\thefootnote}{\fnsymbol{footnote}}
 \footnotetext{}

\section{\bf Introduction}

A general Dirichlet series is a (formal) series of the form $\sum a_n e^{-\lambda_n s}$, where $s
$ is a complex variable, $(a_n)$  a sequence of complex coefficients (called Dirichlet coefficients), and $\lambda=(\lambda_n)$ a frequency
 (a strictly increasing non-negative  real sequence which tends to $+\infty$). Fixing a frequency $\lambda$, we call $D=\sum a_{n}e^{-\lambda_{n}s}$ a $\lambda$-Dirichlet series, and $\mathcal{D}(\lambda)$ denotes the space of all these  series. All basic information on general Dirichlet series can be found in  \cite{HardyRiesz} or \cite{Helson}. In  particular  that convergence
 of  $D=\sum a_{n}e^{-\lambda_{n}s}$  in $s_0 \in \C$ implies convergence in all $s\in \C$
  with $Re s > Re s_0$, and that  the limit function $f(s) = \sum_{n=1}^{\infty} a_{n}e^{-\lambda_{n}s}$ of $D$
  is holomorphic on the half plane $[Re > \sigma_c(D)]$, where
  \[
  \sigma_{c}(D)=\inf\left \{ \sigma \in \R \mid D \text{ converges on } [Re>\sigma] \right\}
  \]
determines the so-called abscissa of convergence.

 \subsection{Helson's theorem}  Let us  start with some details on the state of art of Helson's  result mentioned in the abstract. We first consider the  frequency $\lambda=(\log n)$, which is of special interest, since it generates so-called ordinary Dirichlet series $\sum a_{n}n^{-s}$. As usual (see e.g. \cite{Defant}, \cite{HLS}, or \cite{QQ}), we denote by $\mathcal{H}_2$ the Hilbert space of all Dirichlet series $\sum a_{n}n^{-s}$ with  2-summable coefficients, that is $(a_{n}) \in \ell_{2}$.

Recall that the infinite dimensional polytorus $\T^{\infty}:=\prod_{n=1}^{\infty} \T$
forms a compact abelian group (with its natural group structure), with the normalized Lebesgue measure $dz$ as its Haar measure.
Denote by  $\Xi$ the set of all completely multiplicative characters $\chi\colon \N \to \T$
(that is $\chi(nm)=\chi(n)\chi(m)$ for all $m$,$n$),
which with the pointwise multiplication forms an abelian  group.
Denote by $\mathfrak{p}=(p_{n})$ the sequence of prime numbers.
Looking at the group isomorphism
\begin{align*}
\iota\colon \Xi \to \T^{\infty}, ~~\chi \mapsto (\chi(p_{n})),
\end{align*}
 we see that $\Xi$ also  forms a compact abelian group, and its Haar measure $d \chi$ is the push forward measure of $dz$  through $\iota^{-1}$.

  The following result of
  Helson from \cite{Helson3} (see also
\cite[Theorem 4.4]{HLS}) is our starting point.

 \begin{Theo} \label{HelsonintroHelson}
 Given $D= \sum a_{n}n^{-s}\in \mathcal{H}_2$, for almost all $\chi \in \Xi$ the Dirichlet series
  $D^{\chi} =\sum a_{n} \chi(n) n^{-s}$ converges on the open right half plane $[\text{Re}>0]$.
  \end{Theo}

Helson actually proves an extended version of Theorem \ref{HelsonintroHelson} for general Dirichlet series. Therefore, given a frequency $\lambda$, let us define the space $\mathcal{H}_{2}(\lambda)$ of all (formal) $D=\sum a_{n}e^{-\lambda_{n}s}$ with $2$-summable Dirichlet coefficients. The substitute for $\Xi$ from Theorem \ref{HelsonintroHelson} is given by the so-called Bohr compactification $\overline{\R}$ of $(\R,+)$. Recall that $\overline{\R}$ is a compact abelian group, which may be defined to be the set of all homomorphism $\omega \colon (\R,+) \to \T$  together with the topology of pointwise convergence (i.e. $\overline{\R}$ is the dual group of $(\mathbb{R},+)$ endowed the discrete topology $d$). Additionally, Helson assumes Bohr's condition  $(BC)$ on $\lambda$, that is
\begin{equation} \label{BCHelson} \exists ~l = l (\lambda) >0 ~ \forall ~\delta >0 ~\exists ~C>0~\forall~ n \in \N: ~~\lambda_{n+1}-\lambda_{n}\ge Ce^{-(l+\delta)\lambda_{n}}.
\end{equation}
This condition was isolated by Bohr in \cite{Bohr}, and, roughly speaking it prevents the $\lambda_n$'s from getting too close too fast. Note that $\lambda=(\log n)$ satisfies $(BC)$ with $l=1$.
Then the extended version of  Helson's Theorem~\ref{HelsonintroHelson} reads as follows.

 \begin{Theo} \label{HelsonstheoremHelson}
 Let
$D=\sum a_{n} e^{-\lambda_{n}s} \in \mathcal{H}_{2}(\lambda)$ and $\lambda$ with $(BC)$. Then the Dirichlet series
$D^{\omega}=\sum a_{n} \omega(\lambda_{n}) e^{-\lambda_{n}s}$ converges on $[Re >0]$
for almost all $\omega \in \overline{\R}$.
  \end{Theo}

One of our aims is to extend  Helson's result
to the Hardy space  $\mathcal{H}_{1}(\lambda)$ (a class of Dirichlet series much larger than  $\mathcal{H}_{2}(\lambda)$, see the definition below) under a milder assumption on the frequency $\lambda$. We say that $\lambda$ satisfies
Landau's condition $(LC)$ (introduced in \cite{Landau}) provided
\begin{equation} \label{LCHelson}
\forall~ \delta>0~ \exists ~C>0 ~\forall~ n \in \N \colon~ \lambda_{n+1}-\lambda_{n}\ge C e^{-e^{\delta \lambda_{n}}}.
\end{equation}
Observe that $(BC)$ implies $(LC)$, and that e.g. $\lambda=(\sqrt{\log n})$ satisfies $(LC)$, but fails for $(BC)$. To see an example which fails for $(LC)$, take e.g. $\lambda=(\log \log n)$.

 \subsection{Dirichlet groups}
From \cite{DefantSchoolmann2} we recall the definition and some basic facts of so-called Dirichlet groups.
 Let $G$ be a compact abelian group and $\beta\colon (\R,+) \to G$ a homomorphism of groups.  Then the pair $(G,\beta)$ is called Dirichlet group, if $\beta$ is continuous and has dense range. In this case  the dual map $\widehat{\beta}\colon \widehat{G} \hookrightarrow \R$ is injective, where we identify $\mathbb{R}=\widehat{(\mathbb{R},+)}$ (note that we do not assume $\beta$ to be injective). Consequently, the  characters $e^{-ix\pmb{\cdot}} \colon \R \to \T$, $x\in \widehat{\beta}(\widehat{G})$, are precisely those which define  a unique $h_{x} \in \widehat{G}$  such that $h_{x} \circ \beta=e^{-ix\pmb{\cdot}}$. In particular, we have that
\begin{equation*}
\widehat{G}=\{h_{x} \mid x \in \widehat{\beta}(\widehat{G}) \}.
\end{equation*}
From \cite[Section 3.1]{DefantSchoolmann2} we know that every $L_{1}(\R)$-function may be interpreted as a bounded regular Borel measure on $G$. In particular, for every $u>0$ the Poisson kernel
$$P_{u}(t):=\frac{1}{\pi}\frac{u}{u^{2}+t^{2}}\,,\,\,\, t \in \R,$$
defines a measure $p_{u}$ on $G$, which we call the Poisson measure on $G$. We have $\|p_{u}\|=\|P_{u}\|_{L_{1}(\R)}=1$ and
\begin{equation}\label{Fourier1Helson}
\text{$\widehat{p_{u}}(h_{x})=\widehat{P_{u}}(x)=e^{-u|x|}$ for all $u >0$ and
$x\in \widehat{\beta}(\widehat{G})$.}
\end{equation}
Finally, recall from \cite[Lemma 3.11]{DefantSchoolmann2} that, given  a measurable function $f:G \to \C$, then for almost all $\omega \in G$  there are measurable  functions $f_{\omega} \colon \R \to \C$
such that
\[
\text{$f_{\omega}(t)=f(\omega \beta(t))$ almost everywhere on $\R$,}
\]
and if $f\in L_{1}(G)$, then all these $f_\omega$ are locally integrable.
Moreover, as shown in \cite[Corollary 2.11]{DefantSchoolmann3}, for almost all $\omega \in G$
\begin{equation} \label{besicoHelson}
\widehat{f}(0)=\lim_{T\to \infty} \frac{1}{2T} \int_{-T}^{T} f_{\omega}(t) dt .
\end{equation}

 We will later see, that this way to 'restrict' functions on the group $G$ to $\R$, in fact establishes a sort of  bridge
between  Fourier analysis on Dirichlet groups $(G,\beta)$ and  Fourier analysis on $\R$.

\subsection{$\lambda$-Dirichlet groups}
Now, given a frequency $\lambda$, we call a Dirichlet group $(G,\beta)$ a $\lambda$-Dirichlet group whenever $\lambda \subset \widehat{\beta}(\widehat{G})$, or equivalently whenever for every
$e^{-i\lambda_{n} \pmb{\cdot}} \in \widehat{(\mathbb{R},+)}$ there is (a unique) $h_{\lambda_{n}}\in  \widehat{G}$ with $h_{\lambda_{n}}\circ \beta=e^{-i\lambda_{n} \pmb{\cdot}}$.

 Note that for every $\lambda$ there exists a $\lambda$-Dirichlet groups $(G,\beta)$ (which is not unique).
To see a very first example, take the Bohr compactification $\overline{\R}$ together with the mapping
$$\beta_{\overline{\R}} \colon \R \to \overline{\R}, ~~ t \mapsto \left[ x \mapsto e^{-itx} \right].$$
Then $\beta_{\overline{\R}}$ is continuous and has dense range (see e.g. \cite[Theorem 1.5.4, p. 24]{QQ} or \cite[Example 3.6]{DefantSchoolmann2}), and so the pair $(\overline{\R},\beta_{\overline{\R}})$ forms a $\lambda$-Dirichlet group for all $\lambda$'s. We refer to \cite{DefantSchoolmann2}  for more 'universal' examples of Dirichlet groups. Looking at the  frequency  $\lambda=(n)=(0,1,2,\ldots)$, the group $G=\T$ together with \[\beta_\T: \R \to \T, \,\,\beta_{\T}(t)=e^{-it},\]
forms  a $\lambda$-Dirichlet group, and  the so-called
Kronecker flow
\begin{equation*}
\label{oscarHelson}
\beta_{\T^{\infty}}\colon \R \to \T^{\infty}, ~~ t \mapsto \mathfrak{p}^{-it}=(2^{-it},3^{-it}, 5^{-it}, \ldots),
\end{equation*}
turns the infinite dimensional torus $\T^{\infty}$ into a  $\lambda$-Dirichlet group
for $\lambda = (\log n)$.
We note that, identifying  $\widehat{\T} = \Z$ and $\widehat{\T^\infty} = \Z^{(\N)}$ (all finite sequences of integers), in the first case $h_n(z) = z^n$ for $z \in \T, n \in \Z$,
and in the second case $h_{\sum \alpha_j \log p_j}(z) = z^\alpha$ for  $z \in \T^\infty, \alpha \in \Z^{(\N)}$.

\subsection{Hardy spaces of general Dirichlet series}
Fix some $\lambda$-Dirichlet group $(G,\beta)$ and $1\le p \le \infty$. By
 $$H_{p}^{\lambda}(G)$$
 we denote the  Hardy space of all functions
$f\in L_{p}(G)$ (recall that being a compact abelian group, $G$ allows a unique normalized Haar measure) having a Fourier transform  supported on $\{h_{\lambda_n} \colon n \in \mathbb{N}\} \subset \widehat{G}$. Being a closed subspace of $L_p(G)$, this clearly defines a Banach space.

These spaces $H_{p}^{\lambda}(G)$ naturally define $\lambda$-Dirichlet series. Let
$$\Hcal_{p}(\lambda)$$
be the class of all $\lambda$-Dirichlet series $D=\sum a_n e^{-\lambda_n s}$ for which there is some
$f \in H_p^\lambda(G)$ such that   $a_n = \widehat{f}(h_{\lambda_{n}})$ for all $n$. In this case the function $f$ is unique, and together with
the norm $\|D\|_{p}:=\|f\|_{p}$ the linear space  $\Hcal_{p}(\lambda)$ obviously  forms a Banach space. So (by definition) the so-called Bohr map
\begin{equation} \label{BohrmapHelson}
\Bcal\colon H_{p}^{\lambda}(G)\to \mathcal{H}_{p}(\lambda),~~ f \mapsto \sum \widehat{f}(h_{\lambda_{n}}) e^{-\lambda_{n}s}
\end{equation}
defines an onto isometry. A fundamental fact from   \cite[Theorem 3.24.]{DefantSchoolmann2} is that the definition of $\mathcal{H}_{p}(\lambda)$ is independent of the chosen $\lambda$-Dirichlet group $(G,\beta)$.
Now we have given two definitions of the Hilbert space $\mathcal{H}_{2}(\lambda)$, but  by Parsel's theorem both of these definitions  actually coincide.

Our two basic examples of frequencies, $\lambda = (n)$ and $\lambda = (\log n)$, lead to well-known examples:
\begin{equation} \label{hardyTHelson}
   H_{p}(\T):=H_{p}^{(n)}(\T) \,\,\, \,\text{and} \,\,\,\,   H_p(\T^\infty) := H_p^{(\log n)}(\T^\infty) \,.
\end{equation}
In particular,
$f \in H_p^{(n)}(\T)$ if and only if $f \in L_p(\T)$ and $\hat{f}(n) = 0$ for any $n \in \Z$ with  $n < 0$, and
$f \in H_p^{(\log n)}(\T^\infty)$ if and only if $f \in L_p(\T^\infty)$ and $\hat{f}(\alpha) = 0$ for any finite sequence
$\alpha = (\alpha_k)$ of integers with  $\alpha_k < 0$ for some $k$
(where as usual  $\widehat{f}(\alpha) := \widehat{f}(h_{\log \mathfrak{p}^\alpha}))$.
Consequently, if we turn to Dirichlet series, them the Banach spaces
$$\mathcal{H}_p= \mathcal{H}_p((\log n))$$
are precisely  Bayart's Hardy spaces of ordinary Dirichlet series from \cite{Bayart}
(see also \cite{Defant} and \cite{QQ}).

\subsection{Vertical limits}
Given a $\lambda$-Dirichlet series $D = \sum a_n e^{-\lambda_n s}$ and $z \in \mathbb{C}$, we say that
$$D_{z}:=\sum a_n e^{-\lambda_n z} e^{-\lambda_n s} $$  is the   translation of $D$ about $z$, and we distinguish between horizontal translations $D_u, u \in \R$,  and vertical translations
$D_{i\tau}, \tau \in \R$.

 If $(G, \beta)$ is a $\lambda$-Dirichlet group and $D \in \mathcal{H}_{p}(\lambda)$ is associated to
 $f\in H_p^\lambda(G)$, then for each $u>0$
 the  horizontal  translation $D_u$ corresponds to the convolution of $f$ with the Poisson measure $p_{u}$, i.e.
 $\Bcal(f*p_{u})=D_{u}$ (compare coefficients),
  and we refer to $f*p_{u}$ as the  translation of $f$ about $u$. In particular,  we have that $D_{u}\in \mathcal{H}_{p}(\lambda)$ for every $u>0$.

Moreover, each Dirichlet series of the form $$D^{\omega}:=\sum a_{n} h_{\lambda_{n}}(\omega)e^{-\lambda_n s}\,,\,\, \omega \in G,$$
is said to be a vertical limit of $D$. Examples are vertical translations
$D_{i\tau}$ with $\tau \in \mathbb{R}$,
and the terminology is explained by the fact that each vertical limit  may be approximated by  vertical translates. More precisely, given $D =  \sum a_n e^{-\lambda_n s}$ which converges absolutely on the right half-plane, for every  $\omega \in G$ there is a sequence $(\tau_{k})_{k} \subset \R$ such that $(D_{i\tau_{k}})$ converges to $D^{\omega}$ uniformly on $[Re>\varepsilon]$ for all $\varepsilon>0$.
 Assume conversely that for  $(\tau_{k})_{k} \subset \R$ the  vertical translations $D_{i\tau_k}$ converge
 uniformly on $[Re>\varepsilon]$ for every  $\varepsilon>0$
 to a holomorphic function $f$ on $[Re>0]$. Then there is $\omega \in G$ such that
    $f(s)= \sum_{n=1}^\infty a_n h_{\lambda_n}(\omega) e^{-\lambda_n s}$
    for all $s \in [Re>0]$\,. For all this see \cite[Proposition 4.6]{DefantSchoolmann2}.

\subsection{R\'esum\'e of our results on Helson's theorem}
With all these preliminaries we give a brief r\'esum\'e of our extensions of Helson's theorem \ref{HelsonstheoremHelson}, where we carefully have to distinguish between the cases $1<p<\infty$
and $p=1$.

\medskip

\noindent
{\bf Synopsis I} \label{SIHelson}\\
Let $(G,\beta)$ be a $\lambda$-Dirichlet group, $1 \leq p < \infty$, and $D\in \mathcal{H}_p(\lambda)$
with associated function $f \in H_p^\lambda(G)$. Then the following statements hold true:
\begin{itemize}
\item[(i)]
If $1 < p < \infty$, then  almost all vertical limits $D^\omega$ converge almost everywhere on $[Re = 0]$,
and consequently almost all of them converge on $[Re >0]$.
\item[(ii)]
 If $\lambda$ satisfies $(LC)$ and $p=1$, then  almost all vertical limits $D^\omega$ converge on $[Re >0]$.
\end{itemize}
Moreover,
there is a null set $N \subset G$ such that  for every  $\omega \notin N$ in the first  case
\[
D^\omega(it) = f_\omega (t) \,\,\, \text{for almost all $t \in \R$},
\]
and in both cases
\[
D^\omega(u+it) = (f_\omega \ast P_u) (t)  \,\,\,
\text{for every $u >0$ and almost all $t \in \R$}.
\]

\smallskip

Let us indicate carefully which of these results are already known and which are new. We first discuss the ordinary case $\lambda = (\log n)$ with $(\log n)$-Dirichlet group $(\T^{\infty}, \beta_{\T^\infty})$. Then for $p=2$ statement (i) was proved by Hedenmalm and Saksman in \cite{HedenmalmSaksman}, whereas  Bayart in   \cite[Theorem 6]{Bayart} for every $D \in \mathcal{H}_1$ proves the convergence of almost all vertical limits  $D^\omega$ on $[Re >0]$.
For Dirichlet series in $\mathcal{H}_2$ Bayart deduces his theorem from the Menchoff-Rademacher theorem on almost everywhere convergence of orthonormal series (see also \cite{DefantSchoolmann5}), and extends it then to Dirichlet series
$\mathcal{H}_1$ by  so-called hypercontractivity.
In the general case statement (ii) for $p=2$ is Helson's theorem~\ref{HelsonstheoremHelson} and under the more restrictive condition $(BC)$ instead of $(LC)$ and $p=1$.

 \subsection{Helson's theorem and its maximal inequalities}
 Our strategy is to deduce the preceding results
 \begin{itemize}
 \item
 from relevant maximal inequalities for functions in
$H_{1}^{\lambda}(G)$,
\item
to obtain as a consequence results on pointwise convergence of the Fourier series of these functions,
 \item
 and to use in a final step the  Bohr transform (\ref{BohrmapHelson}) to transfer these results to
Helson-type theorems for  Dirichlet series.
 \end{itemize}
 In the reflexive case $1 < p < \infty$  we follow closely the ideas of Duy \cite{Duy} and Hedenmalm-Saksman \cite{HedenmalmSaksman} extending the Carleson-Hunt theorem on pointwise convergence of Fourier series to functions
in $H_p^\lambda(G)$, and in the non-reflexive case $p=1$ we use among others boundedness properties of a Hardy-Littlewood maximal type  operator for integrable functions on Dirichlet groups which we invent in \cite{DefantSchoolmann3}.

In order to give a r\'esum\'e of the results we have on the first of the above steps recall that given  a measure space $(\Omega, \mu)$ the weak $L_{1}$-space
$L_{1, \infty}(\mu)$ is the linear space of all measurable functions $f\colon \Omega\to \C$ for which  there is a constant $C>0$ such that for all $\alpha>0$ we have
$
\mu \big(\left\{ \omega \in \Omega \mid |f(\omega)|>\alpha \right\} \big)\le
C/\alpha.
$
Together with $\|f\|_{1,\infty}:= \inf C$ the space $L_{1,\infty}(\mu)$ becomes a quasi Banach space (see e.g. \cite[\S 1.1.1 and \S 1.4]{Grafakos1}), where the triangle inequality holds with constant $2$.\\

\noindent
{\bf Synopsis II} \label{SIIHelson}\\
Let $(G,\beta)$ be a $\lambda$-Dirichlet group. Then the following statements hold true:
\begin{itemize}
\item[(i)] For every $1 < p < \infty$ there is a constant $C = C(p) >0$
such that for every $f \in H_p^\lambda(G)$
\begin{equation*}
 \Big\| \sup_{N} \big| \sum_{n=1}^{N} \widehat{f}(h_{\lambda_{n}}) h_{\lambda_{n}} \big|
  \Big\| _{L_p(G)} \le C \,\|f\|_{p}.
\end{equation*}
\item[(ii)]
 If $\lambda$ satisfies $(LC)$, then for every $u >0$ there is a constant $C = C(u) >0$ such that for every $f \in H_1^\lambda(G)$
\begin{equation*}
 \Big\| \sup_{N} \big| \sum_{n=1}^{N} \widehat{f}(h_{\lambda_{n}}) e^{-u \lambda_n} h_{\lambda_{n}} \big|
  \Big\| _{L_{1, \infty}(G)} \le C\,\|f\|_{1}.
\end{equation*}
\item[(iii)]
 If $\lambda$ satisfies $(BC)$,
then to every $u>0$ there is a constant $C = C(u) >0$ such that for all $1\le p \le \infty$ and $f \in H_{p}^{\lambda}(G)$
 \begin{equation*}
  \Big\|\sup_{N} \big| \sum_{n=1}^{N}\widehat{f}(h_{\lambda_{n}})e^{-\lambda_{n}u} h_{\lambda_{n}} \big|
  \Big\|_p \le C\|f\|_{p}.
 \end{equation*}
 \end{itemize}
 In particular, for all $f \in H_p^\lambda(G), 1 <p < \infty$
 \[
 f = \sum_{n=1}^{\infty} \widehat{f}(h_{\lambda_{n}})  h_{\lambda_{n}}
 \,\,\, \text{almost everywhere on $G$},
 \]
and under $(LC)$ for all $f \in H_1^\lambda(G)$ and  $u >0$
\[
f\ast p_u =\sum_{n=1}^{\infty} \widehat{f}(h_{\lambda_{n}}) e^{-u \lambda_n} h_{\lambda_{n}}
\,\,\, \text{almost everywhere on $G$.}
\]

A standard argument shows how to deduce from such maximal inequalities pointwise convergence theorem of Fourier series, e.g. using
Egoroff's theorem (see \cite[Lemma 3.6]{DefantSchoolmann3} for a more general situation). The following remark indicates
 how pointwise convergence theorems of Fourier series then transfer
to Dirichlet series (see  \cite[Lemma 1.4]{DefantSchoolmann3}).

\begin{Rema} \label{tranferHelson}
 Let $(G,\beta)$ be a Dirichlet group, and $f_n, f$ measurable functions on $G$. Then the following are equivalent:
\begin{itemize}
\item[(i)]
$\lim_{n\to \infty} f_n(\omega) = f(\omega)$ \,\,\, \text{for almost all $\omega \in G$.}
\vspace{1mm}
\item[(ii)]
$\lim_{n\to \infty} (f_n)_\omega(t)= f_\omega(t)$ \,\,\,
\text{for almost all $\omega \in G$ and for  almost all $t\in \R$.}
\end{itemize}
In particular, if $(G,\beta)$ be a $\lambda$-Dirichlet group  and $D=\sum a_n e^{-\lambda_n s}$ is associated to $f \in H_1^\lambda(G)$, then
\begin{equation*} \label{FseriesHelson}
f=\sum_{n=1}^{\infty} \widehat{f}(h_{\lambda_{n}})h_{\lambda_{n}}
\end{equation*}
almost everywhere on $G$ if and only if for almost all $\omega \in G$ the Dirichlet series
\begin{equation*} \label{DseriesHelson}
D^{\omega}=\sum a_{n} h_{\lambda_{n}}(\omega) e^{-\lambda_{n}s}
\end{equation*}
converges almost everywhere on the imaginary line $[Re=0]$, and its limit coincides with $f_\omega$
 almost everywhere on $\R$.
\end{Rema}

\subsection{Organization}
The reflexive case from Synopsis I and II we handle in Theorem~\ref{DirichletintegerHelson} and Theorem~\ref{CorointegerHelson}, and under a different point of view also in Theorem~\ref{maximalineqBCHelson}.
The Theorems~\ref{Helson(LC)Helson} and \ref{HelsonstheoHelson} are going to cover the non-reflexive parts. In the final Section~\ref{bohrstheoremsectionHelson} we extend and improve parts of the
structure theory of general Dirichlet series started in \cite{DefantSchoolmann2}. Among others we show in Theorem~\ref{equivalenceHelson} that $\mathcal{D}_{\infty}(\lambda)$, the normed space of all $\lambda$-Dirichlet series which
 converge to a bounded and then holomorphic function on the right half plane,   is complete if and only if $\mathcal{D}_{\infty}(\lambda)=\mathcal{H}_{\infty}(\lambda)$ holds isometrically if and only if
 $\lambda$ satisfies (what we call) 'Bohr's theorem'.

\section{\bf Helson's theorem versus the Carleson-Hunt theorem} \label{CarlesonsectionHelson}
In this section we provide the  proofs of the reflexive  statements  from the Synopses I and II in the introduction.

Therefore, by $CH_{p} >0$ we denote the best constant in the maximal inequality from the Carleson-Hunt theorem -- that is, given $1<p<\infty$, the best $C>0$ such that  for all $f\in L_{p}(\T)$
$$\bigg(\int_{\T} \sup_{N} \big|\sum_{|k|\le N} \widehat{f}(k)z^{k}\big|^{p} dz\bigg)^{\frac{1}{p}}\le C\|f\|_{p}.$$

\begin{Theo} \label{DirichletintegerHelson} Let $1<p<\infty$ and  $\lambda=(\lambda_n)$ an arbitrary frequency. Then for all $\lambda$-Dirichlet group $(G,\beta)$ and $D=\sum a_{n}e^{-\lambda_{n}s}\in \mathcal{H}_{p}(\lambda)$ we for almost all $\omega \in G$ have
\begin{equation}
\label{max1Helson}
\lim_{T\to \infty} \bigg(\frac{1}{2T} \int_{-T}^{T} \sup_{N} \big| \sum_{n=1}^{N} a_{n} h_{\lambda_{n}}(\omega) e^{-it\lambda_{n}} \big|^{p} dt \bigg)^{\frac{1}{p}} \le \text{CH}_{p}\|D\|_{p}.
\end{equation}
Moreover, for almost all $\omega \in G$ almost everywhere on $\R$
\begin{equation}\label{point1Helson}
D^{\omega}(it)=\sum_{n=1}^{\infty}    a_n h_{\lambda_{n}}(\omega) e^{-it\lambda_{n}}=f_{\omega}(t),
\end{equation}
and in particular
\begin{equation}\label{point2Helson}
\text{$D^{\omega}=\sum a_{n} h_{\lambda_{n}}(\omega)e^{-\lambda_{n}s}$ converges on $[Re>0]$.}
\end{equation}
\end{Theo}

\bigskip

As described above
we deduce this from a Carleson-Hunt type maximal inequality for functions in  $H_p^\lambda(G)$.

\begin{Theo} \label{CorointegerHelson} Let $\lambda$ be a frequency and  $1<p<\infty$. Then for all $\lambda$-Dirichlet groups $(G,\beta)$ and $f \in H_{p}^{\lambda}(G)$ we have
\begin{equation} \label{iiiiiHelson}
\bigg( \int_{G} \sup_{N} \big| \sum_{n=1}^{N} \widehat{f}(h_{\lambda_{n}}) h_{\lambda_{n}}(\omega) \big|^{p} d\omega \bigg)^{\frac{1}{p}} \le CH_{p}\|f\|_{p}.
\end{equation}
In particular, almost everywhere on $G$
\begin{equation} \label{sakssHelson}
f=\sum_{n=1}^{\infty} \widehat{f}(h_{\lambda_{n}}) h_{\lambda_{n}}.
\end{equation}
\end{Theo}

Before we begin with the proofs let us apply Theorem~\ref{CorointegerHelson} to the frequency $\lambda =(\log n)$, which, as remarked above, together with the
group $(\T^\infty, \beta_{\T^\infty})$ forms a $(\log n)$-Dirichlet group.

\begin{Coro} \label{OrdiAHelson}
Let $1 <  p < \infty$ and $f \in H_p(\mathbb{T}^\infty)$. Then
\[
\lim_{N\to \infty} \sum_{\mathfrak{p}^\alpha \leq N} \widehat{f}(\alpha) z^\alpha=f(z) \,\,\,\,\, \text{almost everywhere on $\mathbb{T}^\infty$}\,,
\]
and moreover
\[
\bigg( \int_{\mathbb{T}^\infty}   \sup_N \big| \sum_{\mathfrak{p}^\alpha \leq N}
\widehat{f}(\alpha) z^\alpha
\big|^p  d z \bigg)^{1/p}
 \leq  CH_{p} \|f\|_p\,.
\]
\end{Coro}

We start with the proof of Theorem \ref{CorointegerHelson}, and  show  at the end of this section that this result
in fact  also proves Theorem~\ref{DirichletintegerHelson}.

Actually for a certain choice of $\lambda$-Dirichlet groups, Theorem \ref{CorointegerHelson}  is due to  Duy in his article \cite{Duy}, where convergence of Fourier series of so-called Besicovitch almost periodic functions is investigated.

In our language, fixing a frequency $\lambda$, Duy considers the  $\lambda$-Dirichlet group
$G_{D}:=\widehat{(U,d)}$, where  $U$ is the  smallest subgroup of $\R$ containing  $\lambda$ and  $d$ denotes the discrete topology. This compact abelian group together with  the mapping
$$\beta_{D}\colon \R \to G_{D}, ~~ t \mapsto \left[u \mapsto e^{-itu}\right]$$
forms a $\lambda$-Dirichlet group (see also \cite[Example 3.5]{DefantSchoolmann2}). Then by \cite[Theorem 13, p. 274]{Duy} (in our notation) the maximal operator
\begin{equation*} \label{DuymaximalopHelson}
\mathbb{M}(f)(\omega):=\sup_{N>0} \big| \sum_{n=1}^{N} \widehat{f}(h_{\lambda_{n}}) h_{\lambda_{n}}(\omega) \big|
\end{equation*}
defines a bounded operator from $H_{p}^{\lambda}(G_{D})$ to $L_{p}(G_{D})$, whenever $1<p<\infty$, and this in  fact proves  Theorem \ref{CorointegerHelson} for $(G_{D},\beta_{D})$.

Moreover, the case $p=2$ and $\lambda=(\log n)$ with Dirichlet group $(\T^{\infty},\beta_{\T^{\infty}})$ of Theorem \ref{CorointegerHelson} is proven by Hedenmalm and Saksman in \cite[Theorem 1.5]{HedenmalmSaksman}, without stating (\ref{iiiiiHelson}). Their proof and the proof of Duy are based on Carleson's maximal inequality
 on almost everywhere convergence of Fourier series of square integrable functions on $\mathbb{T}$, and
 a technique due to Fefferman from \cite{Fefferman}.

Following closely their ideas, we for the sake of completeness  provide a self-contained proof of Theorem \ref{CorointegerHelson} within our framework of Hardy spaces $H_{p}^{\lambda}(G)$, which shows that  the special  choice of the $\lambda$-Dirichlet group $(G, \beta)$ in fact is irrelevant.

 A crucial argument of \cite{Duy} is, that for every finite set $\{a_{1}, \ldots, a_{N}\}$ of positive numbers, there are $\mathbb{Q}$-linearly independent numbers $b_{1}, \ldots, b_{P}$ such that
$\{\lambda_{1}, \ldots, \lambda_{N}\} \subset \operatorname{span}_{\N_{0}} (b_{1},\ldots, b_{P}).$ We demand for less and only require integer coefficients.

\begin{Lemm} \label{trickDuyHelson} Let $a_{1}, \ldots, a_{N}$ be positive numbers. Then there are $\mathbb{Q}$-linearly independent real numbers $b_{1}, \ldots b_{P}$ such that
$\{a_{1}, \ldots, a_{N}\} \subset \operatorname{span}_{\mathbb{\Z}} (b_{1}, \ldots, b_{P}).$
\end{Lemm}
\begin{proof}
We prove the claim by induction. If $N=1$, then choose $b_{1}:=a_{1}$. Assume that for $a_{1}, \ldots, a_{N}$ there are $\mathbb{Q}$-linearly independent $b_{1}, \ldots, b_{P}$ such that $\{a_{1}, \ldots, a_{N}\} \subset \operatorname{span}_{\mathbb{\Z}} (b_{1}, \ldots, b_{P})$ and let $a_{N+1}$ arbitrary. If $(a_{N+1}, b_{1}, \ldots, b_{P})$ is $\mathbb{Q}$-linearly independent, then
choose $b_{P+1}:=a_{N+1}$. Else, there are rationals $q_{j}$ such that $a_{N+1}=\sum_{j=1}^{P} q_{j}b_{j}$ and so for every $K\in \N$
$$a_{N+1}=\sum_{j=1}^{P} (Kq_{j}) \frac{b_{j}}{K}.$$
Choose $K$ large enough such that $K q_{j}\in \Z$ for all $j$, and define $\widetilde{b_{j}}:=K^{-1}b_{j}$. Then $\{a_{1}, \ldots, a_{N}, a_{N+1}\} \subset \operatorname{span}_{\mathbb{\Z}} (\widetilde{b_{1}}, \ldots, \widetilde{b_{P}})$, which finishes the proof.
\end{proof}

  \begin{proof}[Proof of Theorem \ref{CorointegerHelson}]
We first consider polynomials from $L_{p}(\T^{\infty})$ and then show that the choice of the Dirichlet group is irrelevant. So let $f \in L_{p}(\T^{\infty})$ be a polynomial and define for $x \in \mathbb{R}^N$ the maximal function
  \[
  M_xf(z) = \sup_{S >0}
  \big| \sum_{\substack{\alpha \in \mathbb{Z}^N\\ <\alpha,x> \leq S}} \hat{f}(\alpha) z^\alpha  \big|\,,\,\,\, z \in \mathbb{T}^N\,,
  \]
  where $<\alpha,x>:=\sum \alpha_{j}x_{j}$.   We intend to show that
  \begin{align} \label{maxinequalityHelson}
  \| M_xf\|_p \leq CH_{p}\|f\|_p\,.
  \end{align}

  Note that then,  taking  $x=B$, the proof finishes.
   We will use, that given
a $N \times N$ matrix $M=(m_{i,j})$ with integer entries and such that $\det M =1$, the
transformation formula
for every integrable function
$g: \mathbb{T}^N \to \mathbb{R}$ gives
\begin{align} \label{integralHelson}
\int_{\mathbb{T}^N}  g(z)dz = \int_{\mathbb{T}^N}  g(\Phi_M(z)) dz\,,
\end{align}
 where
 \[
\Phi_M : \mathbb{T}^N \to \mathbb{T}^N\,,  (e^{it_j})_j \mapsto (e^{i \sum_k m_{jk}t_k})_j\,,
\]
and moreover  for all $\alpha \in \mathbb{Z}^N$ and $z \in \mathbb{T}^N$
\begin{align} \label{monoHelson}
\Phi_M(z)^\alpha = z^{M^{t} \alpha}\,,
\end{align}
where $M^{t}$ denotes the transposed matrix of $M$.
    By approximation we only have to prove \eqref{maxinequalityHelson} for a dense collection of $x$ in  $\mathbb{R}_{>0}^N$, and, following the argument from the proof of \cite[Theorem 1.4]{HedenmalmSaksman}, we  take
  \[
  x= \bigg(\frac{q_1}{Q}, \ldots\ldots,\frac{q_N}{Q} \bigg)\,,
  \]
  where $q_1, ., q_n, Q \in \mathbb{Z}$ and $\text{gcd}(q_1,q_2) =1$. Choose  $r_1, r_2 \in \mathbb{Z}$ such that $q_1r_2 - q_2r_1 = 1$, and define the $N \times N$ matrix
  \[
A=
  \begin{bmatrix}
    q_1 & q_2 & q_3 & . & .  &. & . & q_N \\
    r_1 & r_2 & 0 & .& . &.&. & 0 \\
     0 & 0 & 1 & 0 & .  &.&.& 0 \\
     0 & 0 & 0 & 1  & 0 &  . &.& 0 \\
     . &  . &. &.  & .  & .  & .&. &\\
     . &  . &. &.  & .  & .  & .&. &\\
     . &  . &. &.  & .  & .  & .&. &\\
      0 &  0 & 0 & 0  & 0  & . & 0& 1 \\
  \end{bmatrix}
\]
which has determinant one.
Then we deduce from \eqref{integralHelson} and \eqref{monoHelson}
(applied to $M = (A^{-1})^{t}$) that
\begin{align*}
\| M_xf\|_p^p
&
= \int_{\mathbb{T}^N}
\sup_{S >0} \big| \sum_{\substack{\alpha \in \mathbb{Z}^N\\ <q,\alpha> \leq QS}} \hat{f}(\alpha) z^{A^{-1}A\alpha}\big|^p dz
\\&
= \int_{\mathbb{T}^N}
\sup_{S >0} \big| \sum_{\beta \in \{A\alpha \colon  <q,\alpha> \leq QS\}} \hat{f}(A^{-1}\beta) z^\beta\big|^p dz\,.
\end{align*}
Now we obseve that
 for every $S >0$
\begin{align*}\label{indecesHelson}
\{ A\alpha \colon & \text{$\alpha \in \mathbb{Z}^N$  and $<q,\alpha> \leq QS $} \}
=
\{ (\beta_1,\gamma) \in \mathbb{Z} \times\mathbb{Z}^{N-1} \colon \text{$\beta_1 \leq QS $ }\}\,,
\end{align*}
hence
\begin{align*}
\| M_xf\|_p^p
&
= \int_{\mathbb{T}^{N-1}}
\bigg(
\int_{\mathbb{T}}
\sup_{S >0} \big| \sum_{\substack{\beta_1 \in \mathbb{Z}\\ \beta_1 \leq QS}}
\Big[ \sum_{\gamma \in \mathbb{Z}^{N-1} }
\hat{f}(A^{-1}(\beta_1, \gamma)) z^\gamma
 \Big] z_1^{\beta_1}\big|^p dz_1\bigg) dz\,.
\end{align*}
Finally, we deduce from the Carleson-Hunt maximal inequality in $L_p(\mathbb{T}^N)$, and another application of  \eqref{monoHelson} and  \eqref{integralHelson} that
\begin{align*}
\| M_xf\|_p^p
\le \int_{\mathbb{T}^{N-1}}
CH_{p}^{p}
\bigg(
\int_{\mathbb{T}}
\big| \sum_{\beta \in \mathbb{Z}^N}
&
\hat{f}(A^{-1}\beta) z^\beta
\big|^p dz_1\bigg) dz
\\&
=
CH_{p}^{p}
\int_{\mathbb{T}^{N}}
\big| \sum_{\alpha \in \mathbb{Z}^N}
\hat{f}(\alpha) z^\alpha
\big|^p  dz\,,
\end{align*}
which  is what we aimed for. Now let $\lambda$ be a frequency and $(G,\beta)$ be a $\lambda$-Dirichlet group. Fix $N$ and let $E_{N}:=\{\lambda_{1},\ldots \lambda_{N}\}$. Then by Lemma \ref{trickDuyHelson} there are $\mathbb{Q}$-linearly independent $B_{N}:=(b_{1},\ldots, b_{P_{N}})$ such that $E_{N}\subset \operatorname{span}_{\Z} (b_{1},\ldots, b_{P_{N}})$. Let $f=\sum_{n=1}^{N} a_{n}h_{\lambda_{n}}$ and define $g:=\sum c_{\alpha} z^{\alpha} \in L_{p}(\T^{\infty})$, where $c_{\alpha}:=a_{n}$, whenever $\lambda_{n}=\sum \alpha_{j}b_{j}$. Observe that $\T^{P_{N}}$ with mapping
$$\beta_{B_{N}}\colon \R \to \T^{P_{N}}, ~~ t \mapsto (e^{-itb_{1}}, \ldots,e^{-itb_{P_{N}}})$$
forms a Dirichlet group. Then by \cite[Proposition 3.17]{DefantSchoolmann2} we have $\|f\|_{p}=\|g\|_{p}$. Moreover, for every Dirichlet group $(H,\beta_{H})$ we for all $f\in C(H)$ have
\begin{equation}\label{QQQHelson}
\int_{G} f~dm= \lim_{T\to \infty} \frac{1}{2T} \int_{-T}^{T} (f\circ\beta_{H})(t) dt,
\end{equation}
which is straight forward checked on polynomials and follows then by density.
Since $\omega \mapsto \sup_{N\le M} \left| \sum_{n=1}^{N} \widehat{f}(h_{\lambda_{n}}) h_{\lambda_{n}}(\omega) \right|$
is continuous, we obtain using (\ref{QQQHelson}) for $(G,\beta)$ and $(\T^{P_N},\beta_{B_{N}})$ and two times the monotone convergence theorem
\begin{align*}
&\bigg( \int_{G} \sup_{N} \big| \sum_{n=1}^{N} \widehat{f}(h_{\lambda_{n}}) h_{\lambda_{n}}(\omega) \big|^{p} dz \bigg)^{\frac{1}{p}}=\lim_{M\to \infty} \bigg( \int_{G} \sup_{N\le M} \big| \sum_{n=1}^{N} \widehat{f}(h_{\lambda_{n}}) h_{\lambda_{n}}(\omega) \big|^{p} dz \bigg)^{\frac{1}{p}} \\ &= \lim_{M\to \infty} \bigg(\lim_{T\to \infty} \frac{1}{2T} \int_{-T}^{T} \sup_{N\le M} \big| \sum_{n=1}^{N} \widehat{f}(h_{\lambda_{n}})e^{-\lambda_{n}it} \big|^{p} dt \bigg)^{\frac{1}{p}} \\ & =\lim_{M\to \infty}  \bigg( \int_{\T^{\infty}} \sup_{N\le M} \big| \sum_{\alpha B \le N} \widehat{g}(\alpha) z^{\alpha} \big|^{p} dz\bigg)^{\frac{1}{p}}= \bigg( \int_{\T^{\infty}} \sup_{N} \big| \sum_{\alpha B \le N} \widehat{g}(\alpha) z^{\alpha} \big|^{p} dz \bigg)^{\frac{1}{p}}  \\ &\le CH_{p}\|g\|_{p}=CH_{p}\|f\|_{p}. \qedhere
\end{align*}
\end{proof}

\begin{proof}[Proof of Theorem \ref{DirichletintegerHelson}]
 Let $D\in \mathcal{H}_{p}(\lambda)$ and $f\in H_{p}^{\lambda}(G)$ with $\Bcal(f)=D$. By Theorem \ref{CorointegerHelson} we know that
\begin{equation*} \label{fHelson}
\omega \mapsto \sup_{N} \big| \sum_{n=1}^{N} a_{n} h_{\lambda_{n}}(\omega) \big|^p \in L_{1}(G).
\end{equation*}
Then \eqref{besicoHelson} shows  that the  maximal inequality from \eqref{iiiiiHelson} implies the   maximal inequality from \eqref{max1Helson}.
Finally,   \eqref{point1Helson}  is a consequence of \eqref{sakssHelson}  and Remark~\ref{tranferHelson}.
\end{proof}

\section{\bf Helson's theorem  under Landau's condition} \label{maximalineqsectionLCHelson}

It is almost obvious that Theorem~\ref{DirichletintegerHelson}, \eqref{max1Helson} and \eqref{point1Helson}
as well as their  equivalent formulations Theorem~\ref{CorointegerHelson}, \eqref{iiiiiHelson} and \eqref{sakssHelson} of the preceding section fail in the non-reflexive case $p=1$.
Indeed, as described in \eqref{hardyTHelson} we have that  $H_{1}(\T) = H_{1}^{(n)}(\T)$,  and it is  well-known that the Carleson-Hunt theorem fails in  $H_{1}(\T)$.
But as we are going to show now, under Landau's condition $(LC)$ on the frequency  $\lambda$ the Helson-type statement from
Theorem~\ref{DirichletintegerHelson}, \eqref{point2Helson}  can be saved.

\begin{Theo} \label{Helson(LC)Helson}
Let $(G,\beta)$ be a $\lambda$-Dirichlet group for  a frequency $\lambda$ with $(LC)$, and $D = \sum a_n e^{-\lambda_n s}\in \mathcal{H}_1(\lambda)$.
\begin{itemize}
\item[(i)]
Then for   almost all $\omega \in G$ the vertical limits $D^\omega$ converge on $[Re >0]$.
\item[(ii)]
More precisely, there is a null set $N \subset G$ such that  for every  $\omega \notin N$
\[
D^\omega(u+it) = (f_\omega \ast P_u) (t)  \,\,\,
\text{for every $u >0$ and almost all $t \in \R$}\,,
\]
where  $f \in H_1^\lambda(G)$ is  the function associated to $D$ through Bohr's transform.
\end{itemize}
\end{Theo}

As in the preceding section our general setting combined with some of our preliminaries show that this result on general Dirichlet series in fact is equivalent to a
result on pointwise convergence of Fourier series in Hardy spaces on $\lambda$-Dirichlet groups.

\begin{Theo}\label{HelsonstheoHelson}
Let $(G,\beta)$ be a $\lambda$-Dirichlet group for  a frequency $\lambda$ with $(LC)$.
\begin{itemize}
\item[(i)]
 Then for every $u>0$ the sublinear operator
\begin{equation*} \label{operatorHHelson}
S_{max}^{u}(f)(\omega):=\sup_{N} \big|\sum_{n=1}^{N} \widehat{f}(h_{\lambda_{n}})e^{-u\lambda_{n}} h_{\lambda_{n}}(\omega)\big|
\end{equation*}
is bounded from $H_{1}^{\lambda}(G)$ to $L_{1,\infty}(G)$.
\item[(ii)]
Moreover, if $f \in H_{1}^\lambda(G)$, then there is a null set $N\subset G$ such that for every $\omega \notin N$ and every $u>0$ we have
\begin{equation*} \label{guertelHelson}
(f \ast p_{u})(\omega) = \sum_{n=1}^{\infty} \widehat{f}(h_{\lambda_{n}})e^{-u\lambda_{n}} h_{\lambda_{n}}(\omega).
\end{equation*}
\end{itemize}
\end{Theo}

Note that $S_{max}^{u}$ by Theorem~\ref{CorointegerHelson}  without any restriction on $\lambda$ is bounded from $H_{p}^{\lambda}(G)$ to $L_{p}(G)$, whenever  $1<p\le \infty$
(apply Theorem~\ref{CorointegerHelson} for $f \in H_{p}^{\lambda}(G)$ to $f \ast p_u$).

The proof of Theorem~\ref{HelsonstheoHelson}  needs two lemmas, the first one of which in fact is crucial.

\begin{Lemm} \label{jojHelson} Let $\lambda$ be an arbitrary frequency. Then for any sequence $(k_{N})\subset ]0,1]$ the sublinear operator
$$T_{max}(f)(\omega):=\sup_{N} \Big(\big| \sum_{n=1}^{N} \widehat{f}(h_{\lambda_{n}}) h_{\lambda_{n}}(\omega)\big| k_{N}  \Big(\frac{\lambda_{N+1}-\lambda_{N}}{\lambda_{N+1}}\Big)^{k_{N}}\Big)$$
is bounded from $H_{1}^{\lambda}(G)$ to $L_{1,\infty}(G)$ and from $H_{p}^{\lambda}(G)$ to $L_{p}(G)$, where $1<p\le \infty$.
\end{Lemm}
The proof reduces to boundedness properties of the following Hardy-Littlewood maximal type operator $\overline{M}$ introduced in \cite[Section 2.3]{DefantSchoolmann3}:
For  $f\in L_{1}(G)$ and almost all $\omega \in G$ we define
\begin{equation*}
\overline{M}(f)(\omega):=\sup_{I\subset \R} \frac{1}{|I|} \int_{I} |f_{\omega}(t)| dt,
\end{equation*}
where the supremum is taken over all intervals $I\subset \R$.
Then, as shown in \cite[Theorem 2.10]{DefantSchoolmann3}, $\overline{M}$ is  a sublinear bounded operator from $L_{1}(G)$ to $L_{1,\infty}(G)$, and from $L_{p}(G)$ to $L_{p}(G)$, whenever $1<p\le \infty$.
\begin{proof}[Proof of Lemma \ref{jojHelson}]
We recall from \cite[Section 1.3]{DefantSchoolmann3} the notion of Riesz means of some function $f\in H_{1}^{\lambda}(G)$. For $k>0$ and $x>0$ the polynomial
$$R_{x}^{\lambda,k}(f):=\sum_{\lambda_{n}<x} \widehat{f}(h_{\lambda_{n}})\bigg(1-\frac{\lambda_{n}}{x}\bigg)^{k} h_{\lambda_{n}}$$
is called the first $(\lambda,k)$-Riesz mean of $f$. Then, choosing $(k_{N})\subset ]0,1]$, from \cite[Lemma 3.5]{Schoolmann} we know that
\[
\big| \sum_{n=1}^{N} \widehat{f}(h_{\lambda_{n}}) h_{\lambda_{n}}(\omega)\big|\le 3 \bigg(\frac{\lambda_{N+1}}{\lambda_{N+1}-\lambda_{N}} \bigg)^{k_{N}} \sup_{0<x<\lambda_{N+1}} |R^{\lambda,k_{N}}_{x}(f)(\omega)|\,,
\]
 and additionally from \cite[Proposition 3.2]{DefantSchoolmann3} that
\begin{equation*}
\sup_{x>0}|R^{\lambda,k_{N}}_{x}(f)(\omega)|\le CK_{N}^{-1}\overline{M}(f)(\omega),
\end{equation*}
where $C$ is an absolute constant. So together
\begin{equation} \label{bjHelson}
|T_{max}(f)(\omega)|\le 3C \overline{M}(f)(\omega),
\end{equation}
and, since $\overline{M}$ has the stated boundedness properties, the claim follows.
\end{proof}

The second lemma is a standard  consequence of Abel summation.

\begin{Lemm} \label{abeleHelson}
For every $u>0$  there is a constant $C=C(u)$ such that for every choice of complex numbers $a_1, \ldots, a_N$
for all frequencies $\lambda=(\lambda_{n})$ and  $\varepsilon>0$
\[
\big|  \sum_{n=1}^{N} a_n  e^{-(u+\varepsilon) \lambda_n} \big|
\leq
 C(u)\sup_{n\le N}\big|e^{-\varepsilon \lambda_n}\sum_{n=1}^{N} a_n \big|\,.
\]
\end{Lemm}
\begin{proof}
Indeed, by  Abel summation
\begin{align*}
&\big|  \sum_{n=1}^{N} a_n  e^{-(u+\varepsilon) \lambda_{n}} \big|\\ &=\big| e^{-(u+\varepsilon)\lambda_{N}}\sum_{n=1}^{N}a_n + \sum_{n=1}^{N-1} \bigg( \sum_{k=1}^{n} a_n \bigg)(e^{-(u+\varepsilon) \lambda_{n}}-e^{-(u+\varepsilon) \lambda_{n+1}})\big|\\ &\le \sup_{n\le N}\big|e^{-\varepsilon \lambda_{n}}\sum_{k=1}^{n} a_n  \big| \bigg(e^{-u\lambda_{N}}+\sum_{n=1}^{N-1} e^{-u\lambda_{n}}-e^{-u\lambda_{n+1}}e^{-\varepsilon(\lambda_{n+1}-\lambda_{n})} \bigg) \\ &\le
\sup_{n\le N}\big|e^{-\varepsilon \lambda_{n}}\sum_{k=1}^{n} a_n  \big| \bigg(e^{-u\lambda_{N}}+\sum_{n=1}^{N-1} e^{-\varepsilon \lambda_{n}}-e^{-u\lambda_{n+1}} \bigg)\\ &\le \sup_{n\le N}\big|e^{-\varepsilon \lambda_{n}}\sum_{k=1}^{n} a_n  \big| \bigg(1+\frac{1}{u}\int_{0}^{\infty}e^{-ux}dx\bigg)\,\qedhere.
\end{align*}
\end{proof}
\begin{proof}[Proof of Theorem \ref{HelsonstheoHelson}]
For the proof of (i) note first that by $(LC)$ for  every $u>0$ there is a constant $C(u,\lambda) >0$, such that for all $N$
$$\lambda_{N+1}-\lambda_{N} \ge C(u,\lambda) e^{-e^{u\lambda_{N}}}.$$
Hence with the choice $k_{N}:=e^{-u\lambda_{N}}$ we for all $N$ have
\begin{equation}  \label{(A)Helson}
e^{-u\lambda_{N}}\le C_1(u,\lambda) k_{N} \bigg(\frac{\lambda_{N+1}-\lambda_{N}}{\lambda_{N}}\bigg)^{k_{N}} \,,
\end{equation}
and  conclude from   Lemma~\ref{abeleHelson} that
\begin{equation} \label{mittHelson}
S_{\text{max}}^u(f) (\omega)
\leq
  C_{2}(u,\lambda) \sup_{N} \big|e^{-u\lambda_{N}} \sum_{n=1}^{N} \widehat{f}(h_{\lambda_{n}})h_{\lambda_{n}}(\omega)\big|\le C_{3}(u,\lambda) T_{max}(f)(\omega).
\end{equation}
Finally, the  boundedness of $S_{\text{max}}^u: H_1^\lambda(G) \to L_{1,\infty}(G)$ is an immediate  consequence of Lemma~\ref{jojHelson}.

To understand the second statement (ii) take $f \in H_1^\lambda(G)$ and $u >0$.
Then $p_u \ast f \in H_1^\lambda(G)$, and recall from  \eqref{Fourier1Helson} that all non-zero  Fourier coefficients
of this function have the form  $\widehat{f}(h_{\lambda_{n}}) e^{-u \lambda_n}$.
Using a  standard argument (see  again \cite[Lemma 3.6]{DefantSchoolmann3} for a more general situation) gives that there is a null set $N\subset G$ such that on $G\setminus N$ we have
\[f*p_{u}=\sum_{n=1}^{\infty} \widehat{f}(h_{\lambda_{n}})h_{\lambda_{n}}.\]
To finish the proof of (ii)
we need to show that the  dependence of $N$ on  $u>0$ may be avoided: Recall first from  \eqref{mittHelson} and \eqref{bjHelson}
 that for every $u>0$ there is a constant $C(u,\lambda)>0$ which
 for every $f\in H_{1}^{\lambda}(G)$ satisfies
 satisfying
\[S_{max}^{u}(f)(\omega)\le C(u,\lambda) \overline{M}(f)(\omega) \,.\]
So fixing $u>0$ and  $f\in H_{1}^{\lambda}(G)$, we for all $v>0$ obtain that for almost all $\omega$
$$S_{max}^{u+v}(f)(\omega)=S_{max}^{u}(f*p_{v})(\omega)\le C(u,\lambda)\overline{M}(f*p_{v})(\omega)\le C(u,\lambda) \overline{M}(f)(\omega)\,,$$
where the last estimate is taken from \cite[Proof of Proposition 3.7]{DefantSchoolmann3}.
So for all $u>0$ there is a  constant $C_1(u,\lambda)>0$ such that
$$\big\|  \sup_{\alpha\ge u} S_{max}^{\alpha}(f)(\pmb{\cdot}) \big\|_{1,\infty}\le C_1(u,\lambda) \|f\|_{1}
\,\,\,
\text{ and }
\,\,\,
\big\|  \sup_{\alpha\ge u} |f \ast p_\alpha| \big\|_{1,\infty}\le \|f\|_{1}\,,
$$
where the first estimate is a consequence of the $L_{1}$-$L_{1,\infty}$-boundedness of $\overline{M}$ (see again \cite[Theorem 2.10]{DefantSchoolmann3}) and
the second inequality can be found in the proof of \cite[Proposition 2.4]{DefantSchoolmann3}. We  conclude from  \cite[Lemma 3.6]{DefantSchoolmann3}  that for every  $u$ there is a null set $N_{u}\subset G$ such that for all $\omega \notin G$
\begin{equation} \label{hansimglueckHelson}
\lim_{N \to \infty} \sup_{\alpha \ge u} \big| \sum_{n=1}^N \widehat{f}(h_{\lambda_n}) e^{-\alpha \lambda_n} h_{\lambda_n}(\omega) - (f\ast p_\alpha)(\omega)  \big| =0.
\end{equation}
Now collecting all null sets $N_{1/n}, n \in \N,$ gives the conclusion.
\end{proof}

 Now we check that the Helson-type Theorem~\ref{Helson(LC)Helson} is indeed a consequence of the above maximal inequality from Theorem~\ref{HelsonstheoHelson}.

\begin{proof}[Proof of Theorem~\ref{Helson(LC)Helson}]
 Both statements (i) and (ii) follow immediately from (\ref{hansimglueckHelson}) and Remark~\ref{tranferHelson}.
Indeed,  applying Remark \ref{tranferHelson} to (\ref{hansimglueckHelson}) we get that for every $u>0$ there is a null set $N_u\subset G$ such that, if $\omega \notin N_u$, then for almost every $t \in \R$
\begin{equation*}
\lim_{N \to \infty} \sup_{\alpha \ge u} \big| \sum_{n=1}^N \widehat{f}(h_{\lambda_n})e^{-\alpha \lambda_n} h_{\lambda_n}(\omega)e^{-it\lambda_{n}} - (f\ast p_\alpha)(\omega\beta(t))  \big| =0.
\end{equation*}
Hence, again collecting all null sets $N_{1/n}, n \in \N,$ we obtain a null set $N$, such that for every $u>0$ and almost every $t\in \R$
\[D^{\omega}(u+it)=(f*p_{u})(\omega\beta(t))=\int_{\R} f_{\omega}(t-x) P_{u}(x) dx=f_{\omega}*P_{u}(t),\]
whenever $\omega \notin N$, and so the proof is finished.
\end{proof}

\begin{Rema}
Obviously, the  preceding proof of Theorem~\ref{HelsonstheoHelson} works, if we instead of the condition $(LC)$ for $\lambda$ assume that for  every $u>0$ there is a constant $C=C(u)\ge 1$ and sequence $(k_{N})\subset ]0,1]$ such that the estimate from \eqref{(A)Helson} holds for all $N$.
Taking the $k_{N}$th root condition \eqref{(A)Helson} is equivalent to: For every $u>0$ there is a constant $C=C(u)\ge 1$ and sequence $(k_{N})\subset ]0,1]$ such that  for all $N$
\begin{equation*}
\lambda_{N} e^{-u\lambda_{N}k_{N}^{-1}}\bigg(\frac{1}{Ck_{N}}\bigg)^{k_{N}^{-1}}\le \lambda_{N+1}-\lambda_{N}.
\end{equation*}
But then an  elementary calculation shows that this condition  in fact implies $(LC)$.
\end{Rema}

\section{\bf Helson's theorem under Bohr's condition}  \label{maximalineqsectionBCHelson}
We now study the results of the preceding section under the more restrictive condition $(BC)$ instead of $(LC)$ for the frequency  $\lambda$.
We are going to show that under  Bohr's condition $(BC)$  the operator $S_{max}^{u}$ from Theorem \ref{operatorHHelson} improves considerably in the sense that it maps $H_{1}^{\lambda}(G)$ to $L_{1}(G)$ and that its norm is uniformly bounded in $1 \leq p \leq \infty$.

\begin{Theo} \label{HelsonBCinternalHelson}
Let $(BC)$ hold for $\lambda$. Then to every $u>0$ there is a constant $C=C(\lambda,u)$ such that for all $1\le p < \infty$, all $\lambda$-Dirichlet groups $(G,\beta)$  and $D\in \mathcal{H}_{p}(\lambda)$ we for almost all $\omega \in G$ have
\begin{equation*}
\lim_{T\to \infty} \bigg(\frac{1}{2T} \int_{-T}^{T} \sup_{N} \big| \sum_{n=1}^{N} a_{n} h_{\lambda_{n}}(\omega) e^{-(u+it)\lambda_{n}} \big|^{p} dt \bigg)^{\frac{1}{p}} \le C\|D\|_{p}.
\end{equation*}
\end{Theo}

As before we deduce this from an appropriate maximal inequality of 'translated' Fourier series of functions in $H_{p}^{\lambda}(G)$.

\begin{Theo}\label{maximalineqBCHelson} Let $\lambda$ satisfy $(BC)$ and $(G,\beta)$ be a $\lambda$-Dirichlet group. Then for every $u>0$ there is $C=C(u, \lambda)>0$ such that for all $1\le p \le \infty$ and $f \in H_{p}^{\lambda}(G)$
 \begin{equation*}
 \Big\|\sup_{N} \big| \sum_{n=1}^{N}\widehat{f}(h_{\lambda_{n}})e^{-\lambda_{n}u} h_{\lambda_{n}} \big|
  \Big\|_p \le C\|f\|_{p}.
 \end{equation*}
\end{Theo}

Obviously, Theorem~\ref{maximalineqBCHelson} transfers to Theorem~\ref{HelsonBCinternalHelson} precisely as in the proof of Theorem~\ref{DirichletintegerHelson} (given at the end of Section~\ref{CarlesonsectionHelson}).


Let us, as in Corollary~\ref{OrdiAHelson},  apply Theorem \ref{maximalineqBCHelson} to  $\lambda=(\log n)$  and the $\lambda$-Dirichlet group $(\T^{\infty},\beta_{\T^{\infty}})$.

\begin{Coro}\label{OrdiBHelson}
Let  $f \in H_1(\mathbb{T}^\infty)$. Then  for all $u >0$
\[
\lim_{N\to \infty} \sum_{\mathfrak{p}^\alpha \leq N} \widehat{f}(\alpha)\,\Big(\frac{z}{\mathfrak{p}^{u}}\Big)^\alpha=f*p_{u}(z) \,\,\,\,\, \text{almost everywhere on $\mathbb{T}^\infty$}\,,
\]
and moreover
\[
 \int_{\mathbb{T}^\infty}   \sup_N \big| \sum_{\mathfrak{p}^\alpha \leq N} \widehat{f}(\alpha)
\Big(\frac{z}{\mathfrak{p}^{u}}\Big)^\alpha \big|  d z
 \leq C \|f\|_p\,,
\]
where $C= C(u)$ only depends on $u$.
\end{Coro}

Our proof of Theorem \ref{maximalineqBCHelson}, which is inspired by Helson's proof of Theorem~\ref{HelsonstheoremHelson} from  \cite{Helson}, seems to rely strongly  on $(BC)$,
and it requires the following two main ingredients.

\begin{Prop} \label{continuityHelson} Let $1\le p <\infty$, $\varepsilon>0$ and $u>0$. Then the operator
\begin{equation*}
\Psi=\Psi(p,u,\varepsilon) \colon L_{p}(G)\hookrightarrow L_{p}(G,L_{1+\varepsilon}(\R)), ~~ f \mapsto \left[ \omega \mapsto \frac{f_{\omega}*P_{u}}{u+i\pmb{\cdot}} \right]
\end{equation*}
defines a bounded linear embedding with
\begin{equation} \label{normpsiHelson}
\|\Psi\|\le \int_{\R} \bigg(\int_{\R} \bigg( \frac{P_{u}(t-y)}{|u+it|} \bigg)^{1+\varepsilon} dt \bigg)^{\frac{1}{1+\varepsilon}} dy<\infty.
\end{equation}
In particular, if $f \in L_{1}(G)$, then $\frac{f_{\omega}*P_{u}}{u+i\pmb{\cdot}} \in L_{1+\varepsilon}(\R)$ for almost every $\omega \in G$.
\end{Prop}
 So, provided $0<\varepsilon\le 1$, we may apply the Fourier transform $\mathcal{F}_{L_{1+\varepsilon}(\R)}$.
 \begin{Prop} \label{perronHelson} Let $0<\varepsilon\le 1$ and $f \in H^{\lambda}_{1}(G)$. Then we for almost all $\omega \in G$ and for almost all $x \in \R$ have
\begin{equation*}
\mathcal{F}_{L_{1+\varepsilon}(\R)}\bigg(\frac{f_{\omega}*P_{u}}{u+i\pmb{\cdot}}\bigg)(-x)=e^{-u|x|}\sum_{\lambda_{n}<x}\widehat{f}(h_{\lambda_{n}}) h_{\lambda_{n}}(\omega).
\end{equation*}
\end{Prop}
Let us first show how to obtain Theorem \ref{maximalineqBCHelson} from the Propositions \ref{continuityHelson} and \ref{perronHelson}. As already mentioned our strategy is inspired by Helson's proof of Theorem \ref{HelsonstheoremHelson} from \cite{Helson}, which roughly speaking relies on  Plancherel's theorem in $L_{2}(\R)$. Instead
following Helson's ideas we use the Hausdorff-Young inequality in $L_{1+\varepsilon}(\R)$.

\begin{proof}[Proof of Theorem \ref{maximalineqBCHelson}]Adding more entries to the frequency $\lambda$ we may assume that $\lambda_{n+1}-\lambda_{n}\le 1$ for all $n$ (as in the proof \cite[Theorem 4.2]{Schoolmann}). Since $\lambda$ satisfies $(BC)$,  there is $l>0$ and $C=C(\lambda)$ such that $\lambda_{n+1}-\lambda_{n}\ge Ce^{-l\lambda_{n}}$ for all $n$. Let $f \in H_{p}^{\lambda}(G)$. Fix $0<\varepsilon \le 1$ and we choose $q$ such that $\frac{1}{1+\varepsilon}+\frac{1}{q}=1$. By Proposition \ref{continuityHelson} we know that $\frac{P_{u}*f_{\omega}}{u+i\pmb{\cdot}}  \in L_{1+\varepsilon}(\R)$ for almost all $\omega \in G$.  For notational convenience let us define
$$S(f_{\omega})(x)=\sum_{\lambda_{n}<x} \widehat{f}(h_{\lambda_{n}}) h_{\lambda_{n}}(\omega).$$
Then, Proposition \ref{perronHelson} and the Hausdorff-Young inequality imply
\begin{align*}
\infty &>\left\|\frac{P_{u}*f_{\omega}}{u+i\pmb{\cdot}}\right\|_{1+\varepsilon}^{q} \ge \int_{0}^{\infty} |e^{-ut} S(f_{\omega})(t)|^{q} dt=\sum_{n=1}^{\infty} |S(f_{\omega})(\lambda_{n+1})|^{q} \int_{\lambda_{n}}^{\lambda_{n+1}} e^{-uqt} dt  \\ &\ge \sum_{n=1}^{\infty} |S(f_{\omega})(\lambda_{n+1})|^{q}(\lambda_{n+1}-\lambda_{n})e^{-uq\lambda_{n+1}} \ge\sum_{n=1}^{\infty} |S(f_{\omega})(\lambda_{n+1})|^{q} C e^{-l\lambda_{n}}e^{-uq(\lambda_{n}+1)} \\&=Ce^{-uq} \sum_{n=1}^{\infty} |S(f_{\omega})(\lambda_{n+1})|^{q}e^{\lambda_{n}(-uq+l)} \ge Ce^{-uq}  \sup_{N}  |S(f_{\omega})(\lambda_{N+1})|^{q}e^{-\lambda_{N}(uq+l)}\\ &= Ce^{-uq}\sup_{N} \big(|S(f_{\omega})(\lambda_{N+1})| e^{-\lambda_{N}\big(u+\frac{l}{q}\big)} \big)^{q}.
\end{align*}
Hence
$$C^{\frac{1}{q}} e^{-u} \sup_{N} |S(f_{\omega})(\lambda_{N+1})| e^{-\lambda_{N}\big(u+\frac{l}{q}\big)}\le \left\|\frac{P_{u}*f_{\omega}}{u+i\pmb{\cdot}}\right\|_{1+\varepsilon}$$
and therefore with the mapping $\Psi$ from Proposition \ref{continuityHelson}
\begin{align*}
\bigg( \int_{G} \sup_{N } \left| \frac{S(f_{\omega})(\lambda_{N+1})}{e^{\lambda_{N}\big(u+\frac{l}{q}\big)}} \right|^{p} dm(\omega) \bigg)^{\frac{1}{p}} &\le C^{-\frac{1}{q}}e^{u}\bigg( \int_{G} \left\|\frac{P_{u}*f_{\omega}}{u+i\pmb{\cdot}}\right\|_{1+\varepsilon}^{p} dm(\omega) \bigg)^{\frac{1}{p}}\\ &\le C_{1}(u, \lambda)\|f\|_{p} \|\Psi(p,u,\varepsilon)\|.
\end{align*}
Now choosing $\varepsilon$ small enough, such that $l\le q u$, we obtain
with (\ref{normpsiHelson}) from Proposition \ref{continuityHelson}
\begin{equation}\label{rasenHelson}
\bigg( \int_{G} \sup_{N} \left| \frac{S(f_{\omega})(\lambda_{N+1})}{e^{2u\lambda_{N}}} \right|^{p} dm(\omega) \bigg)^{\frac{1}{p}} \le C_{2}(u,\lambda) \|f\|_{p}.
\end{equation}
which together with Lemma \ref{abeleHelson} proves the claim in the range $1\le p <\infty$. Now tending $p$ to $+\infty$ gives the full claim.
\end{proof}
\subsection{Proof of Proposition \ref{continuityHelson}}
The technical part of the proof of Proposition \ref{continuityHelson} is to show that for every $\varepsilon,u>0$
\begin{equation}  \label{monsterHelson}
\int_{\R} \bigg(\int_{\R} \bigg( \frac{P_{u}(t-y)}{|u+it|} \bigg)^{1+\varepsilon} dt \bigg)^{\frac{1}{1+\varepsilon}} dy<\infty.
\end{equation}
 Observe that, if $\varepsilon=0$, then by Fubini's theorem for every $u>0$ this integral is infinity. Since $\|P_{u}\|_{1}=1$ and $\|P_{u}\|_{\infty}=\frac{1}{u}$ by Lyapunov's inequality (see e.g. \cite[Lemma II.4.1, p. 72]{Werner}) we obtain $\|P_{u}\|_{1+\varepsilon} \le \big(\frac{1}{u}\big)^{\frac{\varepsilon}{1+\varepsilon}}$ and so for all $y \in \R$
 \begin{equation} \label{trivialboundHelson}
 \bigg(\int_{\R} \bigg( \frac{P_{u}(t-y)}{|u+it|} \bigg)^{1+\varepsilon} dt \bigg)^{\frac{1}{1+\varepsilon}} \le \frac{1}{u} \|P_{u}\|_{1+\varepsilon}\le \frac{1}{u} \bigg(\frac{1}{u}\bigg)^{\frac{\varepsilon}{1+\varepsilon}}=\bigg(\frac{1}{u}\bigg)^{1+\frac{\varepsilon}{1+\varepsilon}}.
 \end{equation}
Hence the interior integral of (\ref{monsterHelson}) is defined and in order to verify finiteness of (\ref{monsterHelson}) we claim that the interior integral is sufficiently decreasing considered as a function in $y$.
\begin{Lemm} \label{hardHelson} Let $\varepsilon, u>0$. Then we for all $|y|>4u$ have
\begin{equation} \label{uglycalculationHelson}
\bigg(\int_{\R} \bigg( \frac{P_{u}(t-y)}{|u+it|} \bigg)^{1+\varepsilon} dt\bigg)^{\frac{1}{1+\varepsilon}} \le 4|y|^{-\big(1+\frac{\varepsilon}{1+\varepsilon}\big)}.
\end{equation}
In particular,
\begin{equation}
\int_{\R} \bigg(\int_{\R} \bigg( \frac{P_{u}(t-y)}{|u+it|} \bigg)^{1+\varepsilon} dt \bigg)^{\frac{1}{1+\varepsilon}} dy\le 8 \bigg(\frac{1+\varepsilon}{\varepsilon}\bigg) \bigg(\frac{1}{u}\bigg)^{\frac{\varepsilon}{1+\varepsilon}}.
\end{equation}

\end{Lemm}

\begin{proof} Since $|u|+|t|\le 2 |u+it|$, we have
\begin{equation} \label{buchisdaHelson}
\frac{P_{u}(t-y)}{|u+it|}\le 2\frac{P_{u}(t-y)}{u+|t|}.
\end{equation}
Then fixing $y$ we now estimate separately the integrals
$$ (a): ~~\bigg(\int_{0}^{\infty} \bigg( \frac{P_{u}(t-y)}{u+t} \bigg)^{1+\varepsilon} dt \bigg)^{\frac{1}{1+\varepsilon}} ~ \text{and  }~ (b):~~\bigg(\int_{-\infty}^{0} \bigg( \frac{P_{u}(t-y)}{u-y} \bigg)^{1+\varepsilon} dt \bigg)^{\frac{1}{1+\varepsilon}}.$$
Since
$$\int_{-\infty}^{0} \bigg(\frac{P_{u}(t-y)}{u-t}\bigg)^{1+\varepsilon} dt=\int_{0}^{\infty}\bigg(\frac{P_{u}(t+y)}{u+t}\bigg)^{1+\varepsilon} dt=\int_{0}^{\infty}\bigg(\frac{P_{u}(t-(-y))}{u+t}\bigg)^{1+\varepsilon} dt,$$
we see that it suffices to controll integral $(a)$ for $y>0$ and $y<0$. Part I deals with positive $y$ and Part II with negative $y$ in $(a)$.

\textbf{Part I:} Let $y>4u$.  Applying the substitution $x(t)=-y+\frac{1}{t}$ we obtain
\begin{align*}
&\int_{0}^{\infty} \bigg( \frac{u}{(u^{2}+(x-y)^{2})(u+x)} \bigg)^{1+\varepsilon}~ dx\\ &=\int_{0}^{\frac{1}{y}} \bigg( \frac{u}{(u^{2}+(2y-\frac{1}{t})^{2})(u+\frac{1}{t}-y)} \bigg)^{1+\varepsilon} \frac{dt}{t^{2}}\\ &=
\int_{0}^{\frac{1}{y}} |t|^{2\varepsilon} \bigg( \frac{u}{((tu)^{2}+(2yt-1)^{2})(u+\frac{1}{t}-y)} \bigg)^{1+\varepsilon} dt\\ &\le \frac{1}{|y|^{2\varepsilon}}\int_{0}^{\frac{1}{y}}\bigg( \frac{u}{((tu)^{2}+(2yt-1)^{2})(u+\frac{1}{t}-y)} \bigg)^{1+\varepsilon} dt.
\end{align*}
Now we consider the function
$$g(t):=\frac{u}{((tu)^{2}+(2yt-1)^{2})(u+\frac{1}{t}-y)}\,,$$
and we claim that $g$ is strictly increasing
 on $[0,\frac{1}{y}]$ provided $y>4u$. So then
$$\sup_{t \in [0,\frac{1}{y}]}g(t)=g(y^{-1})=\frac{1}{(\frac{u}{y})^{2}+1}\le 1\,,$$
and hence
\begin{equation} \label{beast1Helson}
\bigg(\int_{0}^{\infty} \bigg( \frac{P_{u}(t-y)}{u+t} \bigg)^{1+\varepsilon} dt \bigg)^{\frac{1}{1+\varepsilon}}\le  y^{-\big(\frac{1+2\varepsilon}{1+\varepsilon}\big)}.
\end{equation}
Note that $g$ is not differentiable at $t=\frac{1}{y-u}$. But $g$ is differentiable on $[0, \frac{1}{y}]$, since $\frac{1}{y-u}>\frac{1}{y}$ for $y>u$. We calculate
$$g^{\prime}(t)=\frac{u(-2t^{3}(u-y)(u^{2}+4y^{2})-t^{2}(u^{2}-4uy+8y^{2})+1)}{(t(u-y)+1)^{2}(t^{2}(u^{2}+4y^{2})-4ty+1)^{2}}\,,$$
and show that $g^{\prime}$ is positive. Therefore we only have to focus on the polynomial
$$p(t):=-2t^{3}(u-y)(u^{2}+4y^{2})-t^{2}(u^{2}-4uy+8y^{2})+1.$$
with derivative
\begin{align*}
p^{\prime}(t)&=-6t^{2}(u-y)(u^{2}+4y^{2})-2t(u^{2}-4uy+8y^{2})\\ &=2t(-3t(u-y)(u^{2}+4y^{2})-2(u^{2}-4uy+8y^{2}))),
\end{align*}
which vanishes in $t=0$ and (assuming $y>u$) in
$$t_{0}:=\frac{2(u^{2}-4uy+8y^{2})}{3(y-u)(u^{2}+4y^{2})}.$$
We have $p(0)=1$ and, since $y>4u$, $$p\bigg(\frac{1}{y}\bigg)=\bigg(\frac{u}{y}\bigg)^{2}-2\bigg(\frac{u}{y}\bigg)^{3}-4\bigg(\frac{u}{y}\bigg)+1>0.$$
Moreover $t_{0}>\frac{1}{y}$, and assuming $y>4u$ we have
\begin{align*}
yt_{0}=\frac{2}{3} \frac{8y^{3}-yu(4y-u)}{(y-u)(u^{2}+4y^{2})} \ge \frac{2}{3} \frac{8y^{3}-(y \frac{y}{4}(4y))}{y\big( \big(\frac{y}{4}\big)^{2}+4y^{2} \big)}=  \frac{2}{3} \frac{7}{4+\frac{1}{4}}>1.
\end{align*}
Let us summarize that $p$ is positive on the boundary and has no extremal point in the interior, which implies that $p$ is positive on $[0,\frac{1}{y}]$. Hence $g$ is strictly increasing.

 \textbf{Part II}:
  Now let $y<-4u$. Applying the substitution $x(t)=y+\frac{1}{t}$ we obtain
  \begin{align*}
&\int_{0}^{\infty} \bigg( \frac{u}{(u^{2}+(x-y)^{2})(u+x)} \bigg)^{1+\varepsilon} ~dx =\int_{0}^{-\frac{1}{y}} \bigg( \frac{u}{(u^{2}+(\frac{1}{t})^{2})(u+\frac{1}{t}+y)} \bigg)^{1+\varepsilon} \frac{dt}{t^{2}}\\ &=
\int_{0}^{\frac{1}{|y|}} t^{2\varepsilon} \bigg( \frac{u}{((tu)^{2}+1)(u+\frac{1}{t}+y)} \bigg)^{1+\varepsilon} dt\\ &\le \frac{1}{|y|^{2\varepsilon}}\int_{0}^{\frac{1}{|y|}}\bigg( \frac{u}{((tu)^{2}+1)(u+\frac{1}{t}+y)} \bigg)^{1+\varepsilon} dt.
\end{align*}
We follow the same strategy as before and consider
$$h(t):=\frac{u}{((tu)^{2}+1)(u+y+\frac{1}{t})}.$$
Note that $h$ is differentiable on $[0,\frac{1}{|y|}]$. We calculate
$$h^{\prime}(t)=\frac{-u(t^{3}2u^{2}(u+y)+t^{2}u^{2}-1)}{((tu)^{2}+1)^{2}(t(u+y)+1)^{2}}\,,$$
and  claim that $h$ is increasing on $[0,\frac{1}{|y|}]$. Therefore consider $$p(t)=t^{3}2u^{2}(u+y)+t^{2}u^{2}-1$$
with derivative
$$p^{\prime}(t)=6t^{2}u^{2}(u+y)+2tu^{2}=t2u^{2}(3(u+y)t+1),$$
which vanishes in $t=0$ and in $t_{0}=\frac{-1}{3(u+y)}.$ Note that $t_{0} \in [0,\frac{1}{|y|}]$, whenever $y<-4u$. We have $p(0)=-1$ and $p(\frac{-1}{y})<0$, since
$$p\bigg(\frac{-1}{y}\bigg)=\bigg(\frac{u}{y}\bigg)^{2} \bigg(1-\frac{2(u+y)}{y}\bigg)-1<0,$$
provided $-y>2u$.
Moreover,
$$p(t_{0})=\bigg(\frac{u}{u+y}\bigg)^{2}\bigg(\frac{-2}{27}+\frac{1}{9}\bigg)-1=\frac{1}{27}\bigg(\frac{u}{u+y}\bigg)^{2}-1<0\,,$$
whenever $\bigg(\frac{u}{u+y}\bigg)^{2}\le 27$. But this holds true assuming $y<-4u$, since
$$\bigg(\frac{u}{u+y}\bigg)^{2}\le \bigg(\frac{y}{4}\bigg)^{2} \frac{1}{(-y-(\frac{y}{4}))^{2}}= \frac{1}{9}.$$
Let us summarize, that $p$ is negative on the boundary of $[0,\frac{1}{|y|}]$ and  has a maximum in $t_{0}$ with $p(t_{0})<0$. Hence $p$ is negative on $[0,\frac{1}{|y|}]$, and consequently $h$ is strictly increasing on $[0,\frac{1}{|y|}]$. So we for $y<-4u$ have
\begin{equation} \label{beast2Helson}
\int_{0}^{\infty} \bigg( \frac{u}{(u^{2}+(x-y)^{2})(u+x)} \bigg)^{1+\varepsilon} ~dx \le  |y|^{-2\varepsilon} \int_{0}^{\frac{1}{|y|}}\frac{1}{\frac{u}{|y|}+1} dt\le |y|^{-(1+2\varepsilon)}.
\end{equation}
Hence (\ref{buchisdaHelson}), (\ref{beast1Helson}) and (\ref{beast2Helson}) imply (\ref{uglycalculationHelson}). Moreover with (\ref{uglycalculationHelson}) and (\ref{trivialboundHelson}) we conclude

\begin{align*}
&\int_{\R} \bigg(\int_{\R} \bigg( \frac{P_{u}(t-y)}{|u+it|} \bigg)^{1+\varepsilon} dt \bigg)^{\frac{1}{1+\varepsilon}} dy\\ &=\int_{|y|\le 4u}\bigg(\int_{\R} \bigg( \frac{P_{u}(t-y)}{|u+it|} \bigg)^{1+\varepsilon} dt\bigg)^{\frac{1}{1+\varepsilon}}dy + \int_{|y|>4u} \bigg(\int_{\R} \bigg( \frac{P_{u}(t-y)}{|u+it|} \bigg)^{1+\varepsilon} dt\bigg)^{\frac{1}{1+\varepsilon}} dy \\ &\le
4u\bigg(\frac{1}{u}\bigg)^{1+\frac{\varepsilon}{1+\varepsilon}}+ 4\int_{|y|>4u} |y|^{-\frac{1+2\varepsilon}{1+\varepsilon}} dy=4 \bigg(\frac{1}{u}\bigg)^{\frac{\varepsilon}{1+\varepsilon}}+8\frac{1+\varepsilon}{\varepsilon} \bigg(\frac{1}{4u}\bigg)^{\frac{\varepsilon}{1+\varepsilon}},
\end{align*}
which completes the proof.
\end{proof}

\begin{proof}[Proof of Proposition \ref{continuityHelson}]
Let us for simplicity write
\begin{equation*}
h(y):=\bigg(\int_{\R} \bigg( \frac{P_{u}(t-y)}{|u+it|} \bigg)^{1+\varepsilon} dt \bigg)^{\frac{1}{1+\varepsilon}}.
\end{equation*}
Then applying two times Minkowski's inequality we obtain
\begin{align*}
&\bigg(\int_{G} \left\|\frac{f_{\omega}*P_{u}}{u+i\pmb{\cdot}} \right\|_{1+\varepsilon}^{p} d\omega \bigg)^{\frac{1}{p}}=\bigg( \int_{G}\bigg( \int_{\R} \left| \frac{(f_{\omega}*P_{u})(t)}{u+it} \right|^{1+\varepsilon} dt \bigg)^{\frac{p}{1+\varepsilon}} d\omega \bigg)^{\frac{1}{p}}\\ &= \bigg( \int_{G} \bigg( \int_{\R} \big| \int_{\R} f_{\omega}(y)\frac{P_{u}(t-y)}{u+it} dy \big|^{1+\varepsilon} dt \bigg)^{\frac{p}{1+\varepsilon}} d\omega \bigg)^{\frac{1}{p}} \\ &\le  \bigg(\int_{G} \bigg( \int_{\R} \bigg( \int_{\R} |f_{\omega}(y)|^{1+\varepsilon} \left| \frac{P_{u}(t-y)}{u+it} \right|^{1+\varepsilon} dt \bigg)^{\frac{1}{1+\varepsilon}} dy \bigg)^{p} d\omega \bigg)^{\frac{1}{p}}\\ &=\bigg( \int_{G} \bigg( \int_{\R} |f_{\omega}(y)| h(y) dy \bigg)^{p} d\omega \bigg)^{\frac{1}{p}} \le  \int_{\R} \bigg( \int_{G} |f_{\omega}(y)|^{p} h(y)^{p} d\omega \bigg)^{\frac{1}{p}} dy \\ &\le  \|f\|_{p} \int_{\R} h(y) dy= \|f\|_{p}\int_{\R}\bigg(\int_{\R} \bigg( \frac{P_{u}(t-y)}{|u+it|} \bigg)^{1+\varepsilon} dt \bigg)^{\frac{1}{1+\varepsilon}}dy,
\end{align*}
where the latter integral is finite by Lemma \ref{hardHelson}. Hence $\Psi$ is bounded and defined. To prove injectivity we calculate the Fourier coefficients of $\Psi(f)$. Let first $f=\sum_{n=1}^{N}a_{n} h_{x_{n}}$. Then for all $x\in \R$ and all $t \in \R$

\begin{align*}
\widehat{\Psi(f)}(h_{x})(t)&=\bigg(\int_{G} \Psi(f)(\omega) \overline{h_{x}(\omega)} d\omega\bigg)(t)=\int_{G} \Psi(f)(\omega)(t) \overline{h_{x}(\omega)} d\omega\\ &=\int_{G} \frac{f_{\omega}*P_{u}(t)}{u+it} \overline{h_{x}(\omega)} d\omega= \frac{1}{u+it}\int_{\R} P_{u}(y) \int_{G} f(\omega\beta(t-y)) \overline{h_{x}(\omega)} d\omega dy \\ &=  \frac{1}{u+it}e^{-ixt} \int_{\R} P_{u}(y)e^{iyx} dy \int_{G} f(\eta) \overline{h_{x}(\eta)} d\eta=\frac{1}{u+it}e^{-u|x|}e^{-ixt} \widehat{f}(h_{x}).
\end{align*}
Now by density of polynomials and continuity of $\Psi$ we for all $f\in L_{1}(G)$ obtain
\begin{equation*}
\widehat{\Psi(f)}(h_{x})(t)=\frac{1}{u+it}e^{-u|x|}e^{-ixt} \widehat{f}(h_{x}).
\end{equation*}
Hence, assuming $\Psi(f)=0$, we have $\widehat{\Psi(f)}(h_{x})=0$ and so $\widehat{f}(h_{x})=0$ for all $x$, which implies $f=0$.
\end{proof}

\subsection{Proof of Proposition  \ref{perronHelson}}
To finish the proof of Theorem \ref{maximalineqBCHelson} it remains to calculate the Fourier transform $\mathcal{F}_{L_{1+\varepsilon}(\R)}$ of $\frac{f_{\omega}*P_{u}}{u+i\pmb{\cdot}}$. Observe that this function may fail to be in $L_{1}(\R)$. For instance, if $f=h_{0}$, then $\|\frac{f_{\omega}*P_{u}}{u+i\pmb{\cdot}}\|_{1}=\int_{\R} \frac{1}{|u+it|} dt=\infty$. Our strategy is to calculate for $k>0$ the Fourier transform of $\frac{f_{\omega}*P_{u}}{(u+i\pmb{\cdot})^{1+k}}$ (which belongs to $L_{1}(\R)$) and then we tend $k$ to zero to obtain Proposition \ref{perronHelson}. First we consider polynomials.
\begin{Lemm}\label{kpositiveHelson} Let $g=\sum_{n=1}^{N}a_{n}e^{-i\lambda_{n} \pmb{\cdot}}$ and $k>0$. Then for all $x\in \R$
\begin{equation}
\frac{\Gamma(k+1)}{2\pi} \mathcal{F}_{L_{1}(\R)}\bigg(\frac{g*P_{u}}{(u+i\pmb{\cdot})^{1+k}}\bigg)(-x)=e^{-u|x|} \sum_{\lambda_{n}<x} a_{n}(x-\lambda_{n})^{k},
\end{equation}
where $\Gamma$ denotes the Gamma function.
\end{Lemm}
\begin{proof} From \cite[Lemma 10, p. 50]{HardyRiesz} we have that
for all $\alpha>0$ and $k>0$
\begin{equation}\label{geniusHelson}
\frac{\Gamma(k+1)}{2\pi i}\int_{\alpha-i\infty}^{\alpha+i\infty} \frac{e^{ys}}{s^{1+k}} ds = \begin{cases} y^{k}&, \text{if } y\ge 0,\\ 0 &, \text{if } y<0, \end{cases}
\end{equation}
By linearity it suffices to prove the claim for $g(t)=e^{-\lambda_{n}it}$ for some $n$. Then $g*P_{u}(t)=e^{-(u+it)\lambda_{n}}$ and  we obtain
\begin{align*}
&\frac{\Gamma(k+1)}{2\pi} \mathcal{F}_{L_{1}(\R)}\bigg(\frac{g*P_{u}}{(u+i\pmb{\cdot})^{1+k}}\bigg)(-x)=\frac{\Gamma(k+1)}{2\pi} \int_{\R} \frac{e^{-(u+it)\lambda_{n}}}{(u+it)^{1+k}} e^{xit} dt \\ &=\frac{\Gamma(k+1)}{2\pi} e^{-xu} \int_{\R} \frac{e^{(x-\lambda_{n})(u+it)}}{(u+it)^{1+k}} dt= \frac{\Gamma(k+1)}{2\pi  i} \int_{u-i\infty}^{u+\infty}\frac{e^{(x-\lambda_{n})s}}{s^{1+k}} ds,
\end{align*}
which by (\ref{geniusHelson}) with $\alpha=u$ equals $(x-\lambda_{n})^{k}$, whenever
 $x>\lambda_{n}$, and else vanishes.
\end{proof}

\begin{Lemm} \label{transformpolyHelson} Let $g=\sum_{n=1}^{N}a_{n}e^{-i\lambda_{n} \pmb{\cdot}}$ and $0<\varepsilon\le 1$. Then for almost every $x\in \R$
$$\mathcal{F}_{L_{1+\varepsilon}(\R)}\bigg(\frac{g*P_{u}}{u+i\pmb{\cdot}}\bigg)(-x)=e^{-u|x|} \sum_{\lambda_{n}<x} a_{n}.$$
\end{Lemm}
\begin{proof}
Observe that $\frac{g*P_{u}}{u+i\pmb{\cdot}} \in L_{1+\varepsilon}(\R)$ and $\frac{g*P_{u}}{(u+i\pmb{\cdot})^{1+k}} \in L_{p}(\R)$ for all $k>0$ and $p\ge 1$. The dominated convergence theorem implies $\lim_{k\to 0}\frac{g*P_{u}}{(u+i\pmb{\cdot})^{1+k}}=\frac{g*P_{u}}{u+i\pmb{\cdot}}$ in $L_{1+\varepsilon}(\R)$.
Now by continuity of the Fourier transform and  Lemma \ref{kpositiveHelson}
\begin{align*}
&\mathcal{F}_{L_{1+\varepsilon}(\R)}\bigg(\frac{g*P_{u}}{u+i\pmb{\cdot}}\bigg)=\lim_{k\to \infty}\mathcal{F}_{L_{1}(\R)}\bigg(\frac{g*P_{u}}{(u+i\pmb{\cdot})^{1+k}}\bigg)=\lim_{k\to \infty}\mathcal{F}_{L_{1}(\R)}\bigg(\frac{g*P_{u}}{(u+i\pmb{\cdot})^{1+k}}\bigg)\\ &= C(k)\lim_{k\to 0}   e^{-u|\pmb{\cdot}|} \sum_{\lambda_{n}<\pmb{\cdot}} a_{n}(\pmb{\cdot}-\lambda_{n})^{k}=C(k) e^{-u|\pmb{\cdot}|} \sum_{\lambda_{n}<\pmb{\cdot}} a_{n},
\end{align*}
with $C(k)=\frac{2\pi}{\Gamma(k+1)}$ and convergence in $L_{q}(\R)$, where $\frac{1}{1+\varepsilon}+\frac{1}{q}=1$.
\end{proof}

\begin{proof}[Proof of Proposition \ref{perronHelson}]
Let $(P^{n})$ be a sequence of polynomials from $H_{1}^{\lambda}(G)$ converging to $f$ (see \cite[Proposition 3.14]{DefantSchoolmann2}). Then $\lim_{n\to \infty} \Psi(P^{n})=\Psi(f)$ by Proposition \ref{continuityHelson} and so there is a subsequence $(n_{k})$ such that $\lim_{k\to \infty} \frac{P^{n_{k}}_{\omega}*P_{u}}{u+i\pmb{\cdot}}=\frac{f_{\omega}*P_{u}}{u+i\pmb{\cdot}}$ in $L_{1+\varepsilon}(\R)$ for almost all $\omega \in G$. Hence by continuity of the Fourier transform and Lemma \ref{transformpolyHelson}
\begin{align*}
&\mathcal{F}_{L_{1+\varepsilon}(\R)}\bigg(\frac{f_{\omega}*P_{u}}{u+i\pmb{\cdot}}\bigg)=\lim_{k\to \infty} \mathcal{F}_{L_{1+\varepsilon}(\R)}\bigg(\frac{P^{n_{k}}_{\omega}*P_{u}}{u+i\pmb{\cdot}}\bigg)\\ & =\lim_{k\to \infty} e^{-u|\pmb{\cdot}|} \sum_{\lambda_{n}<\pmb{\cdot}} \widehat{P^{n_{k}}}(h_{\lambda_{n}}) h_{\lambda_{n}}(\omega)=e^{-u|\pmb{\cdot}|} \sum_{\lambda_{n}<\pmb{\cdot}} \widehat{f}(h_{\lambda_{n}}) h_{\lambda_{n}}(\omega). \qedhere
\end{align*}
\end{proof}

\section{\bf Applications} \label{bohrstheoremsectionHelson}

In this final section we give several  applications of the results of the preceding sections.

\subsection{Bohr's theorem and its equivalent formulations}
Suppose that $D=\sum a_{n}e^{-\lambda_{n}s}$ converges somewhere and that its limit function extends  to a bounded and holomorphic function $f$ on $[Re>0]$. Then a prominent problem from the beginning of the 20th century was to determine the class of $\lambda$'s for which under this assumption all $\lambda$-Dirichlet series converge uniformly on $[Re>\varepsilon]$ for every $\varepsilon>0$.

\medskip

{\noindent \bf Bohr's theorem.}
 We say that $\lambda$ satisfies 'Bohr's theorem'  if the answer to the preceding problem is affirmative, and  Bohr indeed proves in \cite{Bohr} that all frequencies with his property $(BC)$ belong to this class.

We denote by $\mathcal{D}^{ext}_{\infty}(\lambda)$ the space of all somewhere convergent $D\in \mathcal{D}(\lambda)$ which have a limit function extending to a bounded and holomorphic functions $f$ on   $[Re>0]$.
It is then immediate that  $\lambda$ satisfies Bohr's theorem if and only if every $D \in \mathcal{D}^{ext}_{\infty}(\lambda)$ converges uniformly on $[Re>\varepsilon]$ for every $\varepsilon>0$.

As  proven in \cite[Corollary 3.9]{Schoolmann},  the    linear space $\mathcal{D}^{ext}_{\infty}(\lambda)$
together with  $\|D\|_{\infty}=\sup_{[Re>0]}|f(s)|$ forms a normed space. The isometric subspace of all $D\in \mathcal{D}^{ext}_{\infty}(\lambda)$, which converge on $[Re>0]$, is  denoted by $\mathcal{D}_{\infty}(\lambda)$.
Note that $\mathcal{D}_{\infty}(\lambda)=\mathcal{D}_{\infty}^{ext}(\lambda)$, whenever Bohr's theorem holds for $\lambda$.

Later in \cite{Landau} Landau improves Bohr's result showing  that the weaker condition $(LC)$ is sufficient for Bohr's theorem.
More generally, we know from  \cite[Remark 4.8.]{Schoolmann} that Bohr's theorem holds for $\lambda$ in each of the following 'testable' cases:

  \begin{itemize}
  \item
   $\lambda$ is $\mathbb{Q}$-linearly independent,
   \item
$L(\lambda):=\limsup_{n\to \infty} \frac{\log n}{\lambda_{n}}=0$,
\item
 $\lambda$ fulfills (LC) (and  in particular, if it fulfills (BC)).
 \end{itemize}
 In particular, the frequency $\lambda =(\log n)$ satisfies Bohr's theorem which constitutes  one of the   fundamental tools within the theory of ordinary Dirichlet series $\sum a_{n}n^{-s}$ (see e.g.  \cite[Theorem 1.13, p. 21]{Defant} or \cite[Theorem 6.2.2., p. 143]{QQ}).

\medskip

{\noindent \bf Completeness.}
In general, $\mathcal{D}_{\infty}(\lambda)$ as well as $\mathcal{D}^{ext}_{\infty}(\lambda)$ may fail to be complete. See \cite[Theorem 5.2]{Schoolmann} for generic example of such  $\lambda$'s.
Let us recall \cite[Theorem 5.1]{Schoolmann}, where we prove that  $\mathcal{D}_{\infty}(\lambda)$ (and consequently also $\mathcal{D}^{ext}_{\infty}(\lambda)$, see Theorem~\ref{equivalenceHelson})  is complete under each of the following concrete  conditions:

  \begin{itemize}
  \item
   $\lambda$ is $\mathbb{Q}$-linearly independent,
   \item
$L(\lambda)=0$,

\item
 $\lambda$ fulfills (LC) and $L(\lambda)<\infty$ (and  in particular, if it fulfills (BC)).
 \end{itemize}

\medskip

{\noindent \bf Coincidence.}
From \cite[Section 2.5]{DefantSchoolmann3} we know that for any $\lambda$ there is an isometric linear map
\begin{equation} \label{isometricembeddingAHelson}
\mathcal{A} \colon \mathcal{D}^{ext}_{\infty}(\lambda) \hookrightarrow H_{\infty}^{\lambda}(G),~~ D\mapsto f
\end{equation}
such that $a_{n}(D)=\widehat{f}(h_{\lambda_{n}})$ for all $n$.
Hence $\mathcal{D}_{\infty}^{ext}(\lambda)$, and so also $\mathcal{D}_{\infty}(\lambda)$, actually are isometric subspaces of $\mathcal{H}_{\infty}(\lambda)$.

Clearly, if $\mathcal{D}_{\infty}(\lambda)$ or $\mathcal{D}^{ext}_{\infty}(\lambda)$ are not complete,
 then $\mathcal{D}_{\infty}^{ext}(\lambda) \varsubsetneq \mathcal{H}_{\infty}(\lambda)$ or $\mathcal{D}_{\infty}(\lambda)  \varsubsetneq \mathcal{H}_{\infty}(\lambda)$, respectively. On the other hand, in the case of the two most prominent examples $\lambda = (n)$ and $\lambda = (\log n)$  we have 'coincidence':
\[
\text{$\mathcal{D}_{\infty}((n))=\mathcal{H}_{\infty}((n))$ and $\mathcal{D}_{\infty}((\log n))=\mathcal{H}_{\infty}((\log n))$;}
\]
the first result is straight forward, the second one a fundamental result from  \cite{HLS} (see also \cite[Corollary 5.3]{Defant}). More generally,
\cite[Theorem 4.12]{DefantSchoolmann2} shows  that we have the isometric 'coincidence'  $\mathcal{D}_{\infty}(\lambda) =  \mathcal{H}_{\infty}(\lambda)$
holds, whenever
  \begin{itemize}
  \item \text{$L(\lambda) < \infty$ and $\mathcal{D}^{ext}_{\infty}(\lambda)=  \mathcal{D}_{\infty}(\lambda)$ (so if e.g. $\lambda$ satisfies Bohr's theorem).}
    \end{itemize}

\medskip

We come to the main point of this subsection -- Bohr's theorem, completeness, and coincidence generate the same class of frequencies.

\begin{Theo} \label{equivalenceHelson} Let $\lambda$ be an arbitrary frequency.
Then the following are equivalent:
  \begin{itemize}
  \item[(a)]
   $\lambda$ satisfies Bohr's theorem,
   \item[(b)]
$\mathcal{D}_{\infty}(\lambda)$ is complete,
\item[(c)]
 $\mathcal{D}_{\infty}(\lambda) =  \mathcal{H}_{\infty}(\lambda)$ isometrically.
 \end{itemize}
 \end{Theo}

Note that  each of the equivalent  statements (a), (b), and (c) of  Theorem \ref{equivalenceHelson}  trivially  implies that $\mathcal{D}_{\infty}(\lambda)=\mathcal{D}^{ext}_{\infty}(\lambda) =  \mathcal{H}_{\infty}(\lambda)$
(look at (c) and \eqref{isometricembeddingAHelson}), and hence
 in this case   $\mathcal{D}^{ext}_{\infty}(\lambda)$ is complete. But we do not now whether in general  completeness of $\mathcal{D}^{ext}_{\infty}(\lambda)$
 implies completeness  of $\mathcal{D}_{\infty}(\lambda)$, which would allow to replace $\mathcal{D}_{\infty}(\lambda)$ in  Theorem \ref{equivalenceHelson}  by $\mathcal{D}^{ext}_{\infty}(\lambda)$. In this context we like to mention, that an example of Neder from \cite{Neder} shows, that in general $D_{\infty}(\lambda)$ is not a closed subspace of $\mathcal{D}_{\infty}^{ext}(\lambda)$.

For the proof of Theorem~\ref{equivalenceHelson} we need some preparation, and start with the following simple consequence of the principle of uniform boundedness.

\begin{Lemm} \label{previousHelson}Assume that $\mathcal{D}_{\infty}(\lambda)$ is complete, and $\varepsilon>0$. Then there is a constant $C=C(\varepsilon)$ such that for all $D \in \mathcal{D}_{\infty}(\lambda)$
\begin{equation*}
\sup_{N}\big\| \sum_{n=1}^{N} a_{n}(D)e^{-\varepsilon\lambda_{n}}e^{-\lambda_{n}s}\big\|_{\infty} \le C\|D\|_{\infty}.
\end{equation*}
\end{Lemm}

\begin{proof}
Define for every $N$
\begin{equation*}
T_{N}(D)=\sum_{n=1}^{N}a_{n}(D)e^{-\varepsilon\lambda_{n}}\colon \mathcal{D}_{\infty}(\lambda)\to \mathbb{C}.
\end{equation*}
Then $T_{N}$ is continuous and $\lim_{N}T_{N}(D)=D(\varepsilon)$ exists. Hence by the principle of uniform boundedness (here completeness of $\mathcal{D}_{\infty}(\lambda)$ is essential) there is a constant $C>0$ such that
\begin{equation*}
\sup_{N} \|T_{N}\|\le C<\infty,
\end{equation*}
that is for all $D\in \mathcal{D}_{\infty}(\lambda)$ we have
\begin{equation} \label{bananeHelson}
\sup_{N} \big| \sum_{n=1}^{N} a_{n}(D)e^{-\lambda_{n}\varepsilon} \big|\le C\|D\|_{\infty}.
\end{equation}
Now let $D\in \mathcal{D}_{\infty}(\lambda)$. Applying (\ref{bananeHelson}) to $D_{z}$, which belong to $\mathcal{D}_{\infty}(\lambda)$ for all $z\in [Re>0]$,   we obtain
\begin{equation*}
\sup_{z\in [Re>0]} \sup_{N} \big| \sum_{n=1}^{N} a_{n}e^{-\lambda_{n}z}e^{-\lambda_{n}\varepsilon} \big|\le C \sup_{z\in [Re>0]}\|D_{z}\|_{\infty}\le C\|D\|_{\infty},
\end{equation*}
which proves the claim.
\end{proof}

The  second lemma  is  crucial, and in fact a consequence of the Helson-type Theorem \ref{DirichletintegerHelson}  (compare this with \cite[Propositions 4.3 and 4.5]{DefantSchoolmann2}).

\begin{Lemm} \label{complHelson} Let $\lambda$ be an arbitrary frequency and  $D\in \mathcal{H}_{\infty}(\lambda)$. Then for every  $\lambda$-Dirichlet group  $(G,\beta)$
almost all vertical limits $D^\omega \in  \mathcal{D}_{\infty}(\lambda)$ and  $\|D^{\omega}\|_{\mathcal{D}_{\infty}(\lambda)}=\|D\|_{\mathcal{H}_{\infty}(\lambda)}$.
\end{Lemm}
\begin{proof} Let  $f \in H_{\infty}^{\lambda}(G)$ be the function associated to $D$, i.e.  $\Bcal(f)=D$. Since $\mathcal{H}_{\infty}(\lambda)\subset \mathcal{H}_{2}(\lambda)$ and the function $f_{\omega}*P_{u}$ is continuous, Theorem \ref{DirichletintegerHelson} implies that $D^{\omega}$ converges on $[Re>0]$ and $D^{\omega}(u+it)=f_{\omega}*P_{u}(t)$ for all $t\in \R$ and $u>0$.
Hence
\[\sup_{[Re>u]} |D^{\omega}(s)|=\sup_{[Re=u]} |D^{\omega}(s)|\le \|f_{\omega}*P_{u}\|_{\infty}\le \|f_{\omega}\|_{\infty} \leq   \|f\|_{\infty},\]
and so $D^{\omega}\in \mathcal{D}_{\infty}(\lambda)$ with $\|D^{\omega}\|_{\mathcal{D}_{\infty}(\lambda)}\le \|f\|_{\infty} = \|D\|_{\mathcal{H}_{\infty}(\lambda)}$.
Moreover, by  \cite[Propositions 4.3]{DefantSchoolmann2} and \eqref{isometricembeddingAHelson}  we have that $\|D\|_{\mathcal{H}_{\infty}(\lambda)}= \|D^{\omega}\|_{\mathcal{H}_{\infty}(\lambda)}
= \|D^\omega\|_{\mathcal{D}_{\infty}(\lambda)}$.
\end{proof}

The  third and final ingredient we need for the proof of Theorem~\ref{equivalenceHelson} is a 'Bohr-Cahen formula' for the abscissa of uniform convergence
for general Dirichlet series. Given a Dirichlet series $D=\sum a_{n}e^{-\lambda_{n}s}$, the abscissa $\sigma_{u}(D)$ of uniform convergence
is defined to be the infimum over all $\sigma \in \R$ such $D$ converges uniformly on $[Re > \sigma]$. The following   convenient estimate for $\sigma_{u}(D)$ is proved in \cite[Corollary 2.5]{Schoolmann}:
\begin{equation}\label{alwaysBohrcahenHelson}
\sigma_{u}(D)\le \limsup_{N\to \infty} \frac{\log\bigg( \sup_{t\in \R} \big|\sum_{n=1}^{N} a_{n}e^{-it\lambda_{n}}\big|\bigg)}{\lambda_{N}}.
\end{equation}

\begin{proof}[Proof of Theorem~\ref{equivalenceHelson}]
In a first step we prove the equivalence $(b) \Leftrightarrow (c)$: Obviously, $(c)$ implies $(b)$. So assume that $(b)$ holds, and let $D\in \mathcal{H}_{\infty}(\lambda)$. Then $D^{\omega}\in \mathcal{D}_{\infty}(\lambda)$ for some $\omega \in G$ by Lemma \ref{complHelson}. Applying \cite[Proposition 3.4, $k=1$]{Schoolmann} for every $\varepsilon>0$ the Dirichlet polynomials
\begin{equation*}
R_{x}(D^{\omega}_{\varepsilon})=\sum_{\lambda_{n}<x} a_{n}(D)e^{-\varepsilon\lambda_{n}} h_{\lambda_{n}}(\omega) \bigg(1-\frac{\lambda_{n}}{x}\bigg) e^{-\lambda_{n}s}
\end{equation*}
converge uniformly to $D^{\omega}$ on $[Re>0]$. Hence, by \cite[Corollary 4.4]{DefantSchoolmann2} (Dirichlet polynomials
in $\mathcal{D}_\infty(\lambda)$ and their vertical limits have the same norm)

\begin{equation*}
R_{x}(D_{\varepsilon})=\sum_{\lambda_{n}<x} a_{n}(D)e^{- \varepsilon\lambda_{n}} \bigg(1-\frac{\lambda_{n}}{x}\bigg) e^{-\lambda_{n}s}, ~x>0,
\end{equation*}
define a Cauchy net in $\mathcal{D}_{\infty}(\lambda)$.  Then $(R_{x}(D_{\varepsilon}))$ by $(b)$ has a limit in $\mathcal{D}_{\infty}(\lambda)$, which is $D_{\varepsilon}$ with $a_{n}(D_{\varepsilon})=a_{n}(D)e^{-\varepsilon\lambda_{n}}$ for all $n$ and $\|D_{\varepsilon}\|_{\mathcal{D}_{\infty}(\lambda)}\le \|D\|_{\mathcal{H}_{\infty}(\lambda)}$ for all $\varepsilon>0$.
Hence, as desired,  $D \in \mathcal{D}_{\infty}(\lambda)$.

In a second step, we check that $(a) \Leftrightarrow (c)$, and start with the implication $(a) \Rightarrow (c)$.
So let  again $D \in \mathcal{H}_{\infty}(\lambda)$. We have to show that $D \in \mathcal{D}_{\infty}(\lambda)$.
By Lemma \ref{complHelson} there is some $\lambda$-Dirichlet group $(G, \beta)$ and some $\omega \in G$ such that $D^{\omega}\in \mathcal{D}_{\infty}(\lambda)$
and $\|D\|_{\mathcal{H}_\infty(\lambda)} = \|D^\omega\|_{\mathcal{D}_\infty(\lambda)} $. We denote by $D_N$ the $N$th partial sum of $D_N$, and by $D_{N,\varepsilon}$
its horizontal translation by $\varepsilon >0$. Then, for every $\varepsilon>0$, assuming Bohr's theorem for $\lambda$, the sequence $(D^{\omega}_{N,\varepsilon})$ converges to $D^{\omega}_\varepsilon$ in $\mathcal{D}_\infty(\lambda)$. By   \cite[Corollary 4.4]{DefantSchoolmann2} we know that
\begin{equation*}
\sup_{N}\sup_{t\in \R} | D_{N,\varepsilon}(it)|=  \sup_{N}\sup_{t\in \R} | D^\omega_{N,\varepsilon}(it)|<\infty,
\end{equation*}
which by \eqref{alwaysBohrcahenHelson} implies that $\sigma_{u}(D)\le 0$.  So $D$ converges on the right half-plane, and it  remains to show that the limit function of $D$ is bounded on all of $[Re >0]$.
Indeed, if $\varepsilon>0$, then for  large $N$ (again by \cite[Corollary 4.4]{DefantSchoolmann2})
\begin{align*}
\|D_{N,\varepsilon}\|_{\mathcal{D}_\infty(\lambda)}
&
= \|D^\omega_{N,\varepsilon}\|_{\mathcal{D}_\infty(\lambda)}
\\&
 \leq 1 + \|D^\omega_{\varepsilon}\|_{\mathcal{D}_\infty(\lambda)}
\leq  1+ \|D^\omega\|_{\mathcal{D}_\infty(\lambda)} =  1+ \|D\|_{\mathcal{H}_\infty(\lambda)}\,.
\end{align*}
Hence $\|D\|_{\mathcal{D}_\infty(\lambda)} \leq 1 + \|D \|_{\mathcal{H}_\infty(\lambda)} < \infty$, the conclusion.

Assume conversely that $(c)$ holds, that is,   $\mathcal{D}_\infty(\lambda) = \mathcal{H}_\infty(\lambda)$. Then
 $\mathcal{D}_\infty(\lambda)$ is complete and by \eqref{isometricembeddingAHelson} we have  $\mathcal{D}_{\infty}^{ext}(\lambda)=\mathcal{D}_{\infty}(\lambda)$.
In order to check $(a)$ take  some  $D\in \mathcal{D}_{\infty}^{ext}(\lambda)$; we have to show that $\sigma_{u}(D)\le 0$.
 Indeed, by Lemma~\ref{previousHelson} and another application of the Bohr-Cahen formula~\eqref{alwaysBohrcahenHelson} we know that $\sigma_{u}(D_\varepsilon)\le 0 $ for all $\varepsilon>0$, which implies  $\sigma_{u}(D)\le 0$.
\end{proof}

\begin{Rema} \label{extHelson} A simple analysis of the previous proof shows that the equivalence (b) and (c) of Theorem \ref{equivalenceHelson} holds true, if we replace $\mathcal{D}_{\infty}(\lambda)$ by $\mathcal{D}_{\infty}^{ext}(\lambda)$, that is for any frequency $\lambda$ we have that
  $\mathcal{D}^{ext}_{\infty}(\lambda)$ is complete if and only if $\mathcal{D}_{\infty}^{ext}(\lambda)=\mathcal{H}_{\infty}(\lambda)$. Indeed, if we assume that $\mathcal{D}^{ext}_{\infty}
(\lambda)$ is complete, then in particular for $\varepsilon=1$ the sequence $(R_{x}(D_{1}))$ has a limit $D_{1}\in \mathcal{D}_{\infty}^{ext}(\lambda)$. Hence $\sigma_{c}(D_{1})<\infty$, which implies $\sigma_{c}(D)<\infty$ and so $D\in \mathcal{D}_{\infty}^{ext}(\lambda)$. Again, we do not know whether completeness of  $\mathcal{D}^{ext}_{\infty}(\lambda)$ implies, that $\lambda$ satisfies Bohr's theorem.
\end{Rema}

Let us apply  Theorem \ref{equivalenceHelson} to the concrete frequency $\lambda = (\sqrt{\log n})$ which obviously satisfies $(LC)$, so fulfills Bohr's theorem. Then, although in this case
$L((\sqrt{\log n}))=+\infty$ (!), we may conclude the  following (apparently non-trivial) application.

\begin{Coro}
$\mathcal{D}_{\infty}((\sqrt{\log n})) = \mathcal{H}_{\infty}((\sqrt{\log n}))$, and $\mathcal{D}_{\infty}((\sqrt{\log n}))$ is complete.
\end{Coro}

\smallskip
\subsection{Norm  of the partial sum operator  in $\pmb{\mathcal{H}_{\infty}(\lambda)}$}

Recall from above that Bohr's theorem holds for $\lambda=(\log n)$, and that a quantitative variant of this (see again \cite[Theorem 6.2.2., p. 143]{QQ} or \cite[Theorem 1.13, p. 21]{Defant}) reads as follows: There is a constant $C>0$ such that for every $D\in \mathcal{D}_{\infty}((\log n))$ and $N$
\begin{equation} \label{ordinaryBohrquantiHelson}
\big\| \sum_{n=1}^{N} a_{n} n^{-s} \big\|_{\infty} \le C\log(N) \|D\|_{\infty}.
\end{equation}
Given an arbitrary frequency $\lambda$, we are interested
 in establishing quantitative  variants of Bohr's theorem in the sense of (\ref{ordinaryBohrquantiHelson}), and this means to control the norm of the partial sum operator
\begin{equation*}
S_{N}\colon \mathcal{D}^{ext}_{\infty}(\lambda) \to \mathcal{D}_{\infty}(\lambda), ~~ D \mapsto \sum_{n=1}^{N}a_{n}(D) e^{-\lambda_{n}s}.
\end{equation*}
The main result of \cite[Theorem 3.2]{Schoolmann} is then, that for all $0<k\le 1$, $D=\sum a_{n}e^{-\lambda_{n}s}\in \mathcal{D}^{ext}_{\infty}(\lambda)$ and $N$ we have
\begin{equation} \label{normofSNHelson}
\big\|\sum_{n=1}^{N} a_{n}e^{-\lambda_{n}s}\big\|_{\infty}\le C \frac{\Gamma(k+1)}{k} \bigg(\frac{\lambda_{N}}{\lambda_{N+1}-\lambda_{N}}\bigg)^{k} \|D\|_{\infty},
\end{equation}
where $C$ is an absolute constant and $\Gamma$ denotes the Gamma function. The case $p=\infty$ of Lemma \ref{jojHelson}  extends (\ref{normofSNHelson}) from $\mathcal{D}^{ext}_{\infty}(\lambda)$ to $\mathcal{H}_{\infty}(\lambda)$.

\begin{Theo} \label{coro22Helson} Let $\lambda$ be an arbitrary freuency. Then for all $D\in \mathcal{H}_{\infty}(\lambda)$, all $0<k\le 1$ and all $N$ we have
\begin{equation*}
\big\|\sum_{n=1}^{N} a_{n}(D) e^{-\lambda_{n}s} \big\|_{\infty} \le \frac{C}{k} \bigg(\frac{\lambda_{N+1}}{\lambda_{N+1}-\lambda_{N}}\bigg)^{k} \|D\|_{\infty},
\end{equation*}
where $C>0$ is a universal constant.
\end{Theo}
In particular, assuming $(LC)$ (respectively, (BC)) for $\lambda$ and choosing $k_{N}=e^{-\delta\lambda_{N}}$ (respectively, $k_{N}=\lambda_{N}^{-1}$) we deduce from Theorem \ref{coro22Helson}
(see also again \eqref{(A)Helson}) the following quantitative variants of Bohr's theorem in $\mathcal{H}_{\infty}(\lambda)$. See \cite[Section 4]{Schoolmann} for the corresponding results for $\mathcal{D}^{ext}_{\infty}(\lambda)$.
\begin{Coro} \label{coro1Helson} Let $(LC)$ hold for $\lambda$. Then to every $\delta>0$ there is a constant $C=C(\delta)$ such that for all $D\in \mathcal{H}_{\infty}(\lambda)$ and $N$
\begin{equation*}
\big\|\sum_{n=1}^{N} a_{n}(D) e^{-\lambda_{n}s} \big\|_{\infty} \le Ce^{\delta\lambda_{N}} \|D\|_{\infty}.
\end{equation*}
If $\lambda$ satisfies (BC), then for every $D\in \mathcal{H}_{\infty}(\lambda)$ and $N$
\begin{equation*}
\big\|\sum_{n=1}^{N} a_{n}(D)e^{-\lambda_{n}s}\big \|_{\infty} \le C_{1}\lambda_{N} \|D\|_{\infty}.
\end{equation*}
with an absolute constant $C_{1}>0$.
\end{Coro}
\begin{proof}[Proof of Theorem \ref{coro22Helson}] Let us for simplicity write $C=C(k,N):=\frac{1}{k}\bigg(\frac{\lambda_{N+1}}{\lambda_{N+1}-\lambda_{N}}\bigg)^{k}.$
Then for all $\omega \in G$ with $T_{\max}$ from Lemma \ref{jojHelson} we have
\begin{align*}
\big|\sum_{n=1}^{N} \widehat{f}(h_{\lambda_{n}}) h_{\lambda_{n}}(\omega)\big|=C C^{-1}\big|\sum_{n=1}^{N} \widehat{f}(h_{\lambda_{n}})h_{\lambda_{n}}(\omega)\big|\le C T_{\max}(f)(\omega),
\end{align*}
and so the claim follows, since $T_{\max}\colon H_{\infty}^{\lambda}(G) \to L_{\infty}(G)$ is bounded.
\end{proof}

\smallskip

\subsection{Montel theorem}
In the case of ordinary Dirichlet series, so series with frequency $\lambda=(\log n)$, Bayart in \cite{Bayart} (see also \cite[Theorem 3.11]{Defant} or \cite[Theorem 6.3.1]{QQ}) proves an important  Montel-type theorem in $\mathcal{H}_{\infty}=\mathcal{D}_{\infty}((\log n))$: For every bounded sequence $(D^{j})$ in $\mathcal{D}_{\infty}((\log n))$ there are  a subsequence $(D^{j_{k}})$ and $D\in \mathcal{D}_{\infty}((\log n))$ such that $(D^{j_{k}})$ converges uniformly to $D$ on $[Re>\varepsilon]$ for every $\varepsilon>0$.

Bayart's Montel theorem extends to $\lambda$-Dirichlet series whenever $\lambda$ satisfies (LC) or  $L(\lambda)=0$, or is $\mathbb{Q}$-linearly independent (see \cite[Theorem 4.10]{Schoolmann}). Moreover, as proven in \cite[Theorem 4.19]{DefantSchoolmann2}, under one of the the assumptions [(LC) and $L(\lambda)<\infty$, or $L(\lambda)=0$, or  $\mathbb{Q}$-linear independence] it extends from $\mathcal{D}_{\infty}(\lambda)$ to $\mathcal{H}_{p}(\lambda)$.

We prove a considerable extension of all this.
A consequence of Theorem \ref{equivalenceHelson} shows that Bayart's Montel theorem holds for every frequency $\lambda$ which satisfies Bohr's theorem (or equivalently (b) or (c) from Theorem \ref{equivalenceHelson}).

\begin{Theo} \label{MontelHelson} Assume that Bohr's theorem holds for $\lambda$ and  $1\le p \le \infty$. Then for every bounded sequence $(D^{j})$ in $\mathcal{H}_{p}(\lambda)$ there is a subsequence $(D^{j_{k}})_{k}$ and $D\in \mathcal{H}_{p}(\lambda)$ which converges to $D_{\varepsilon}$ in $\mathcal{H}_{p}(\lambda)$ for every $\varepsilon>0$. The same result holds true, if we replace $\mathcal{H}_{p}(\lambda)$ by $\mathcal{D}_{\infty}(\lambda)$.
\end{Theo}

We follow the same strategy as in the proof of \cite[Theorem 4.19]{DefantSchoolmann2}. We first prove Theorem \ref{MontelHelson} for $\mathcal{D}_{\infty}(\lambda)$, and then, using some vector valued arguments, we extend this result to  $\mathcal{H}_{p}(\lambda)$.

Therefore, let us recall, that, given a frequency $\lambda$ and a Banach space $X$, we denote by $\mathcal{D}_{\infty}(\lambda,X)$ the linear space of all Dirichlet series $D=\sum a_{n}e^{-\lambda_{n}s}$ which have  coefficients $(a_{n})\subset X$ and which  converge and  define a bounded function on $[Re>0]$ (then being holomorphic and with values in $X$).
 A result from \cite{vectorvalued} states that for any non-trivial Banach space $X$, the space $\mathcal{D}_{\infty}(\lambda)$ is complete if and only if  $\mathcal{D}_{\infty}(\lambda,X)$ is complete (again  endowed with sup norm on $[Re >0]$).

 Moreover, a standard Hahn-Banach argument shows that Lemma~\ref{previousHelson} extends  from the scalar-valued case to the vector-valued case:
Given  $\varepsilon >0$,  there is a constant $C=C(\varepsilon) >0$ such that for
every   Banach space $X$ and every $D=\sum a_{n}e^{-\lambda_{n}s} \in \mathcal{D}_{\infty}(\lambda,X)$
\begin{equation} \label{translationinftyHelson}
\sup_{N}\big\| \sum_{n=1}^{N} a_{n}e^{-\varepsilon\lambda_{n}}e^{-\lambda_{n}s}\big\|_{\infty} \le C\|D\|_{\infty}\,,
\end{equation}
provided that $\mathcal{D}_{\infty}(\lambda)$ is complete, or equivalently $\lambda$ satisfies
 Bohr's theorem (Theorem \ref{equivalenceHelson}). Indeed, apply  Lemma~\ref{previousHelson} to the Dirichlet series
 $x^\ast \circ D=\sum x^{\ast}(a_{n})e^{-\lambda_{n}s} \in \mathcal{D}_{\infty}(\lambda),\, x^\ast \in X^\ast$, and use a standard Hahn-Banach argument.

\begin{proof}[Proof of Theorem \ref{MontelHelson}] We first assume that  $p=\infty$, so that by assumption and Theorem \ref{equivalenceHelson}  we have that $\mathcal{D}_{\infty}(\lambda)=\mathcal{H}_{\infty}(\lambda)$. Moreover, we at first look at a bounded sequence  $(D^{j})$  in $\mathcal{D}_{\infty}(\lambda)$, and denote the  coefficients of $D^{j}$ by $(a_{n}^{j})_{n}$. So, by \cite[Corollary 3.9]{Schoolmann} there is a constant $C>0$ such that for all $n$, $j$
\begin{equation}
|a_{n}^{j}|\le \|D^{j}\|_{\infty}\le \sup_{j} \|D^{j}\|_{\infty}\le C <\infty.
\end{equation}
Hence by a diagonal process we find a subsequence $(j_{k})_{k}$ such that $\lim_{k\to \infty} a^{j_{k}}_{n}=:a_{n}$ exists for all $n$. Moreover, applying \eqref{translationinftyHelson} we obtain for every $\varepsilon>0$
a constant $C_1 = C_1(\varepsilon)>0$ such that for all $N$
\begin{equation*}
\sup_{k} \big\| \sum_{n=1}^{N} a_{n}^{j_{k}} e^{-\varepsilon\lambda_{n}}e^{-\lambda_{n}s}\big\|_{\infty}\le C_{1}\sup_{k}\|D^{j_{k}}\|_{\infty}<C_{1}C<\infty\,.
\end{equation*}
 Hence with $D=\sum a_{n}e^{-\lambda_{n}s}$, by \cite[Proposition 2.4]{Schoolmann} we obtain that $(D^{j_{k}}_{\varepsilon})$ converges uniformly to $D_{\varepsilon}$ on $[Re>\delta]$ for every $\delta>0$, which proves the claim for $\mathcal{D}_{\infty}(\lambda)$. Now let $1\le p < \infty$ and $(D^{j})$ a bounded sequence in $\mathcal{H}_{p}(\lambda)$. Since $\mathcal{D}_{\infty}(\lambda)$ is complete under Bohr's theorem (Theorem \ref{equivalenceHelson}), by \cite[Lemma 4.9]{DefantSchoolmann2} the map
\begin{equation*} \label{embeddingptoDirichletHelson}
\Phi \colon \mathcal{H}_{p}(\lambda) \hookrightarrow \mathcal{D}_{\infty}(\lambda,\mathcal{H}_{p}(\lambda)), ~~ \sum a_{n}e^{-\lambda_{n}s} \mapsto \sum (a_{n}e^{-\lambda_{n}z}) e^{-\lambda_{n}s}
\end{equation*}
defines an into isometry. Hence $(\Phi(D^{j}))$ is a bounded sequence in $\mathcal{D}_{\infty}(\lambda,\mathcal{H}_{p}(\lambda))$ and again for all $n$, $j$
\begin{equation*}
|a_{n}^{j}|=\|a_{n}^{j}e^{-\lambda_{n}z}\|_{p}\le \|\Phi(D^{j})\|=\|D^{j}\|_{p}\le \sup_{j} \|D^{j}\|_{p} \le C<\infty\,,
\end{equation*}
for some absolute constant $C>0$. By another  diagonal process we obtain a subsequence $(j_{k})_{k}$ such that $\lim_{k\to \infty} a_{n}^{j_{k}}=:a_{n}$ exists, and using  \eqref{translationinftyHelson} together with the
vector-valued variant of \cite[Proposition 2.4]{Schoolmann} (its proof follows word by word from the scalar case) we conclude, that
$(\Phi(D^{j_{k}}_{\varepsilon}))$ converges in $\mathcal{D}_{\infty}(\lambda,\mathcal{H}_{p}(\lambda))$ for every $\varepsilon>0$  as $k\to \infty$. Hence, the sequence $(D_{\varepsilon}^{j_{k}})$ forms a Cauchy sequence in $\mathcal{H}_{p}(\lambda)$ with limit $D_{\varepsilon}$, and the proof is complete.
\end{proof}

\smallskip

\subsection{Nth abschnitte}
Let $H_{\infty}(B_{c_{0}})$ denote the Banach space of of all holomorphic and bounded functions on the open unit ball $B_{c_{0}}$ (of the Banach space $c_{0}$ of all zero sequences). Then as proven in \cite{HLS}
 (see also \cite[Theorem 5.1]{Defant}) there is an isometric bijection
\begin{equation} \label{HLSHelson}
H_{\infty}(B_{c_{0}}) \to H_{\infty}(\T^{\infty}), ~~ F\mapsto f,
\end{equation}
which preserves the Taylor and Fourier coefficients in the sense that $c_{\alpha}(F)=\widehat{f}(\alpha)$ for all multi indices $\alpha$.

Recall, that $F\colon B_{c_{0}} \to \C$ belongs to $H_{\infty}(B_{c_{0}})$ if and only if $F$ is continuous and all its restrictions $F_{N}\colon \mathbb{D}^{N} \to \mathbb{C}$ belong to $H_{\infty}(\mathbb{D}^{N})$ with $\sup_{N} \|F_{N}\|_{\infty}<\infty$ (see e.g. \cite[Corollary 2.22]{Defant}). By the Bohr map (\ref{BohrmapHelson}) and (\ref{HLSHelson}) this result transfers to ordinary Dirichlet series: A Dirichlet series $D=\sum a_{n}n^{-s}$ belongs to $\mathcal{D}_{\infty}((\log n))$ if and only if for every $N$ its so-called $N$th abschnitt, that is $D|_{N}=\sum a_{n}n^{-s}$, where the sum is taken  over all natural numbers which only have the first $N$ prime numbers as divisors, belong to $\mathcal{D}_{\infty}((\log n))$ with $\sup_{N} \|D|_{N}\|_{\infty} <\infty$ (see also \cite[Corollary 3.10]{Defant}).

This result extends to general Dirichlet series. To understand this let us recall, that for  every frequency $\lambda$ there is another real sequence $B=(b_{n})$ such that for  every $n$ there are finitely many rationals $q_{1}^{n}, \ldots q_{k}^{n}$ such that
\begin{equation*}
\lambda_{n}=\sum q_{j}^{n} b_{n}.
\end{equation*}
In this case, we call  $B$  basis, and $R=(q^{n}_{j})_{n,j}$  Bohr matrix of $\lambda$. Moreover, we write $\lambda=(R,B)$, whenever $\lambda$ decomposes with respect to a basis $B$ with Bohr matrix $R$, and note that every $\lambda$ allows a subsequence which is a basis $B$ for $\lambda$.

Suppose that $\lambda=(R,B)$ and let $D\in \mathcal{D}(\lambda)$. Then the Dirichlet series $D|_{N}=\sum a_{n}(D)e^{-\lambda_{n}s}$,
where $a_{n}(D)\ne 0$ implies that $\lambda_{n}\in \operatorname{span}_{\mathbb{Q}}(b_{1},\ldots, b_{N})$, is denoted as the $N$th abschnitt of $D$.

A consequence of Theorem \ref{MontelHelson} gives an improvement of \cite[Theorem 4.22]{DefantSchoolmann2}.

\begin{Theo} Assume that Bohr's theorem holds for $\lambda$, $1\le p \le \infty$ and $D=\sum a_{n}e^{-\lambda_{n}s}$. Then $D\in \mathcal{H}_{p}(\lambda)$ if and only if
its $N$th abschnitte $D|_{N} \in \mathcal{H}_{p}(\lambda)$  with $\sup_{N} \|D|_{N}\|_{p}<\infty$. Moreover, in this case
$\|D\|_{p}=\sup \|D|_{N}\|_{p}$, and the same results holds true, whenever  we replace $\mathcal{H}_{p}(\lambda)$ by $\mathcal{D}_{\infty}(\lambda)$.
\end{Theo}
\begin{proof}
The 'if part' precisely is Remark 4.21 from \cite{DefantSchoolmann2}, and holds true without any assumption on $\lambda$. So, suppose $D|_{N} \in \mathcal{H}_{p}(\lambda)$ for all $N$ with $\sup_{N} \|D|_{N}\|_{p}<\infty$. Then by Theorem \ref{MontelHelson} there is a subsequence $(N_{k})$ and $E\in \mathcal{H}_{p}(\lambda)$ such that $(D_{1}|_{N_{k}})$ converges to $E_{1}$ as $k \to \infty$. Comparing Dirichlet coefficients we see, that $a_{n}(E)e^{-\lambda_{n}}=a_{n}(E_{1})=a_{n}e^{- \lambda_{n}}$ and so $E=D$.
\end{proof}

\end{document}